\documentclass[11pt]{article}
\usepackage[utf8]{inputenc}
\def\final{1}
\usepackage{color}
\usepackage{subfigure}
\usepackage{url}
\usepackage{graphicx}
\usepackage{amsmath,amssymb}
\usepackage{amsthm}
\usepackage{bm}
\usepackage[plainpages=false]{hyperref}

\usepackage{titlesec}
\titleformat{\section}
{\normalfont\Large\bfseries\filcenter}{\thesection.}{1 ex}{}
\titleformat{\subsection}
{\normalfont\normalsize\bfseries}{\thesubsection}{1 ex}{}
\titleformat{\subsubsection}[runin]
{\normalfont\normalsize\bfseries}{\thesubsubsection.}{1 ex}{}

\usepackage[margin=1.2 in]{geometry}

\usepackage[charter]{mathdesign}
\usepackage[mathcal]{eucal}

\usepackage[square,numbers,sort&compress]{natbib}

\ifnum\final=0
\newcommand{\mynote}[1]{\marginpar{\tiny\sf #1}}
\else
\newcommand{\mynote}[1]{}
\fi

\usepackage{thmtools,thm-restate}
\declaretheorem{theorem}
\declaretheorem[within=section]{lemma}
\declaretheorem[sibling=lemma,name=Proposition]{prop}

\newcommand{\figref}[1]{Figure \ref{fig:#1}}

\newcommand{\lemref}[1]{Lemma \ref{lemma:#1}}
\newcommand{\lemrefX}[1]{\ref{lemma:#1}}
\newcommand{\propref}[1]{Proposition \ref{prop:#1}}
\newcommand{\theoref}[1]{Theorem \ref{theo:#1}}
\newcommand{\secref}[1]{Section \ref{sec:#1}}

\newcommand{\secrefX}[1]{\ref{sec:#1}}

\newcommand{\lemlab}[1]{\label{lemma:#1}}
\newcommand{\proplab}[1]{\label{prop:#1}}
\newcommand{\theolab}[1]{\label{theo:#1}}
\newcommand{\seclab}[1]{\label{sec:#1}}

\newcommand{\Gal}{(\tilde{G},\varphi,\tilde{\bm{\ell}})}
\newcommand{\Gad}{(\tilde{G},\varphi,\tilde{\vec d})}
\newcommand{\Gtilde}{\tilde{G}}
\newcommand{\elltilde}{\tilde{\bm{\ell}}}

\newcommand{\bgamma}{\bm{\gamma}}
\newcommand{\Gpa}{G(\vec p,\Phi)}
\newcommand{\Vtilde}{\tilde{V}}
\newcommand{\Etilde}{\tilde{E}}
\newcommand{\Euc}{\operatorname{Euc}}
\newcommand{\into}{\hookrightarrow}

\newcommand{\Trans}{\Lambda}

\newcommand{\Teich}{\operatorname{Teich}}

\newcommand{\JMat}[1]{ \left( \begin{array}{rr} #1 \end{array} \right)}

\newcommand{\Id}{\operatorname{Id}}
\newcommand{\rep}{\operatorname{rep}}

\newcommand{\teich}{\operatorname{teich}}

\newcommand{\GammaTT}{$\Gamma$-$(2,2)$ }
\newcommand{\GammaOO}{$\Gamma$-$(1,1)$ }
\newcommand{\GammaCL}{$\Gamma$-colored-Laman }
\newcommand{\cent}{\operatorname{cent}}
\newcommand{\Cent}{\operatorname{Cent}}

\newcommand{\Rep}{\operatorname{Rep}}
\newcommand{\Rad}{\operatorname{cl}}

\renewcommand{\vec}[1]{\mathbf{#1}}

\newcommand{\iprod}[2]{\left\langle {#1},{#2}\right\rangle}

\newcommand{\R}{\mathbb{R}}
\newcommand{\Z}{\mathbb{Z}}

\newcommand{\eop}{\hfill$\qed$}
\begin{document}
\title{Frameworks with forced symmetry II: \\
Orientation-preserving crystallographic groups}
\author{Justin Malestein\thanks{Einstein Institute of Mathematics, Hebrew University of Jerusalem, justinmalestein@gmail.com}
\and Louis Theran\thanks{Institut für Mathematik,
Diskrete Geometrie, Freie Universität Berlin, theran@math.fu-berlin.de}}
\date{}
\maketitle
\begin{abstract}
We give a combinatorial characterization of minimally rigid planar frameworks with
orientation-preserving crystallographic symmetry, under the constraint of forced symmetry.
The main theorems are proved by extending the methods of the first paper in this sequence
from groups generated by a single rotation to groups generated by translations and
rotations.  The proofs make use of a new family of matroids defined on crystallographic
groups and associated submodular functions.
\end{abstract}

\section{Introduction} \seclab{intro}

A \emph{crystallographic framework} is an \emph{infinite} planar structure, \emph{symmetric} with respect to a crystallographic
group, made of fixed-length bars connected by universal joints with full rotational freedom.  The allowed continuous
motions preserve the lengths and connectivity of the bars (as in the finite framework case) and
\emph{symmetry with respect to the group $\Gamma$} (this is the new addition).
However, the representation of $\Gamma$ is \emph{not} fixed and may change. This
model extends the one from \cite{BS10}, using a formalism similar to \cite{MT13,MT11,MT12a,MT12fsi}.
Figures \ref{fig:gam2graphtoperiodic} and \ref{fig:gam4graphtoperiodic} show examples of crystallographic frameworks.
\begin{figure}[htbp]
\centering
\subfigure[]{\includegraphics[width=.35\textwidth]{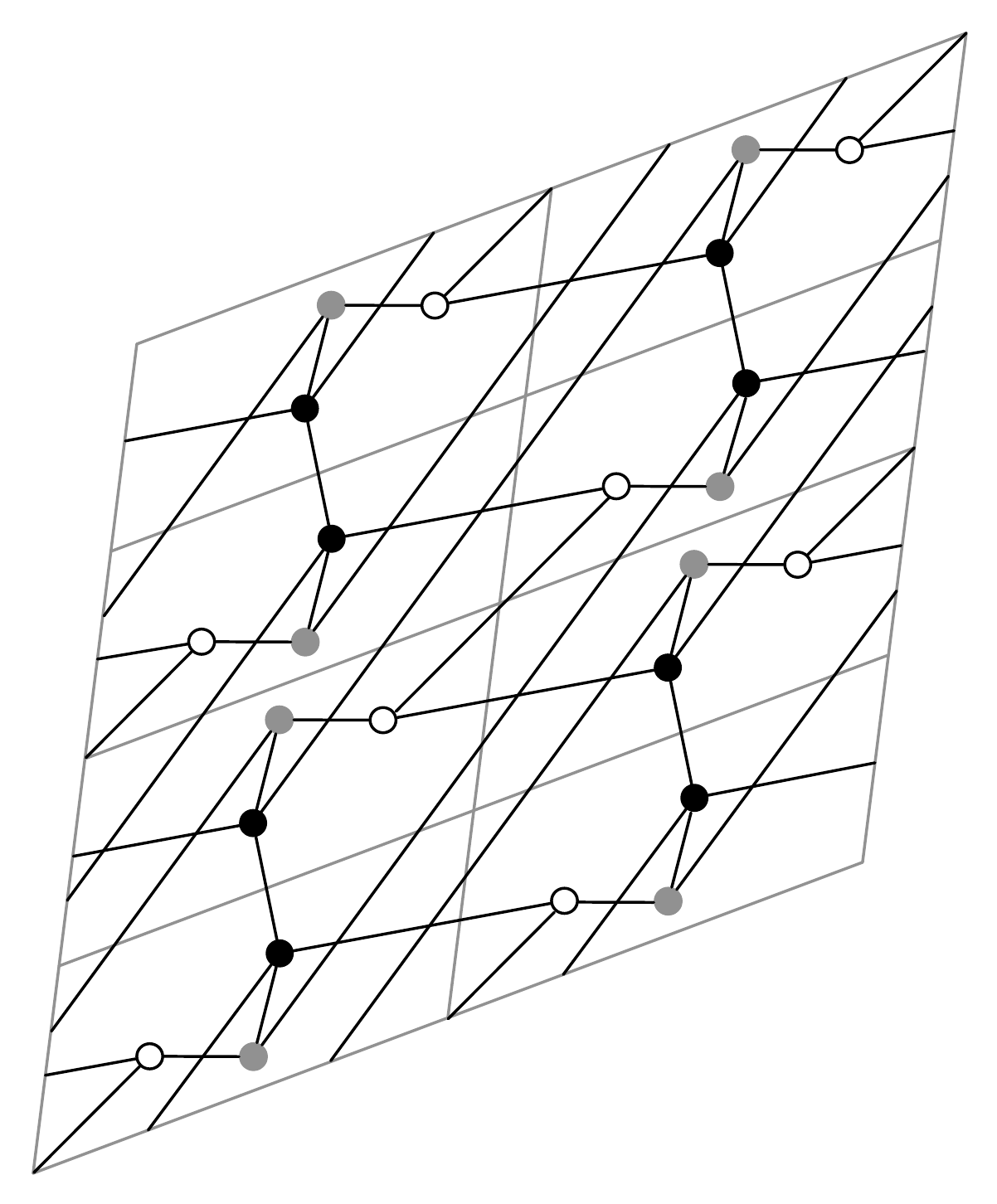}}
\subfigure[]{\includegraphics[width=.35\textwidth]{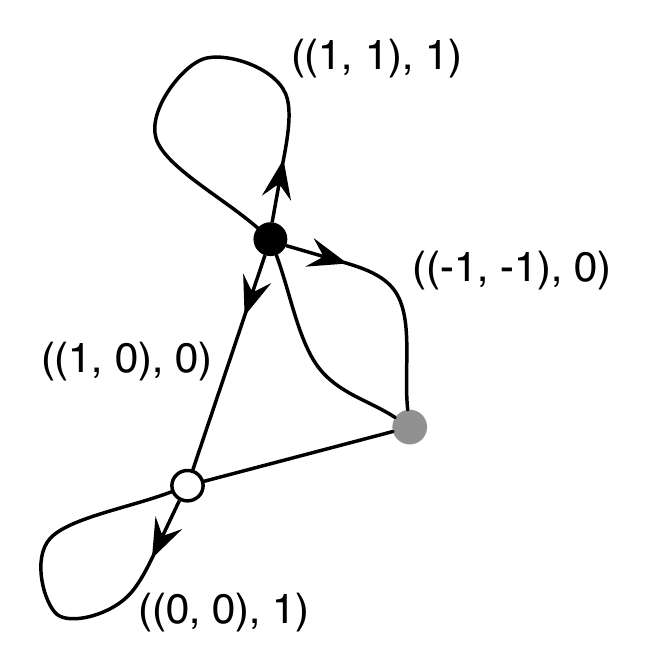}}
\caption{A $\Gamma_2$-crystallographic framework:
(a) A piece of an infinite crystallographic framework with $\Gamma_2$ symmetry.  The group
$\Gamma_2$ is generated by an order $2$ rotation and translations.  The origin, which is
a rotation center, is at the center of the diagram.  Each quadrilateral (with gray edges) is a
fundamental domain of the $\Gamma_2$-action on $\R^2$.
(b) The associated \emph{colored graph} capturing the underlying combinatorics.  Edges that
are not marked and oriented are colored with the identity element of $\Gamma_2$.  The vertices
in (b) are shaded differently to show the fibers over each of them in (a).}
\label{fig:gam2graphtoperiodic}
\end{figure}
A crystallographic framework is \emph{rigid} when the only allowed motions (that, additionally,
must act on the representation of $\Gamma$) are Euclidean isometries  and \emph{flexible} otherwise.

\begin{figure}[htbp]
\centering
\subfigure[]{\includegraphics[width=0.35\textwidth]{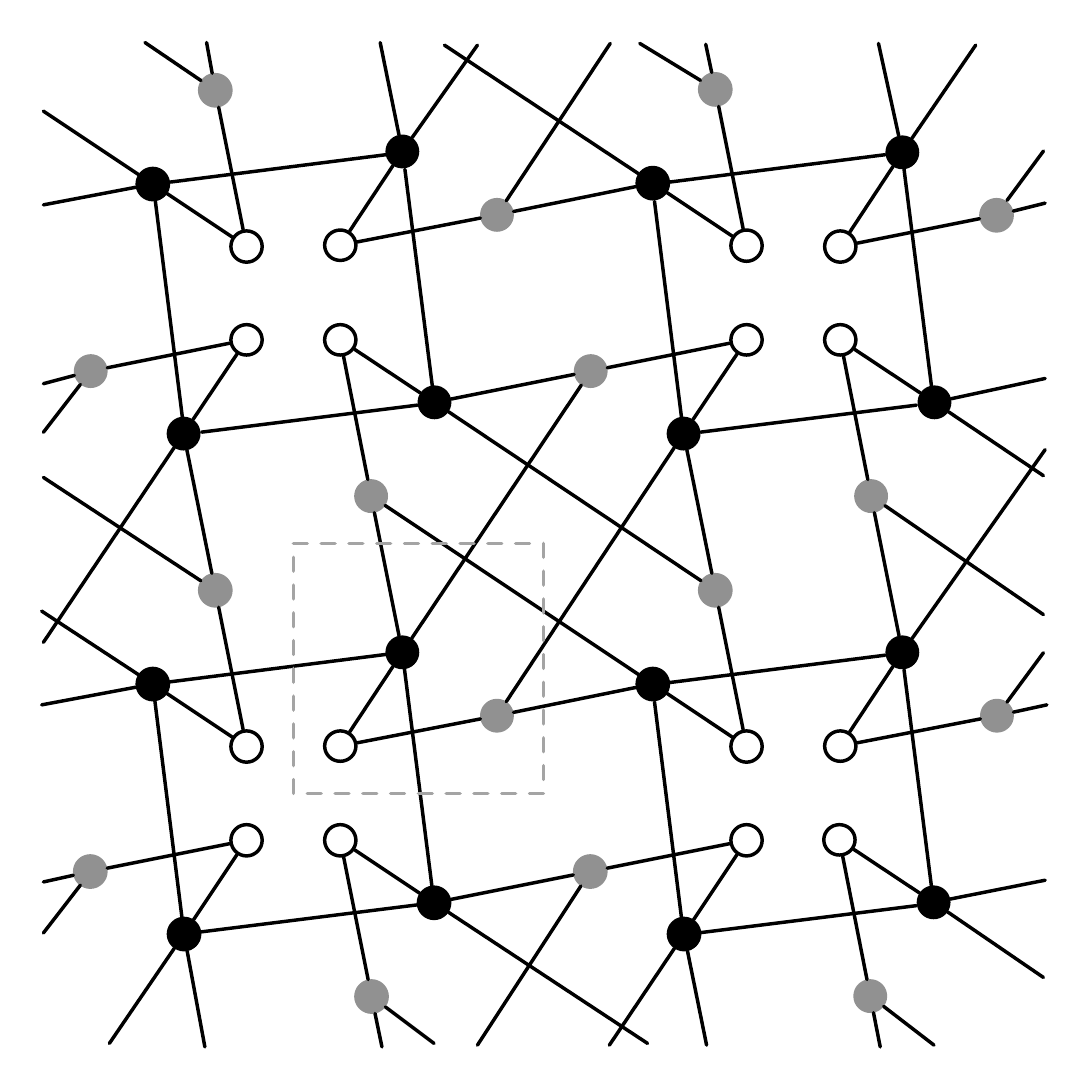}}
\subfigure[]{\includegraphics[width=0.35\textwidth]{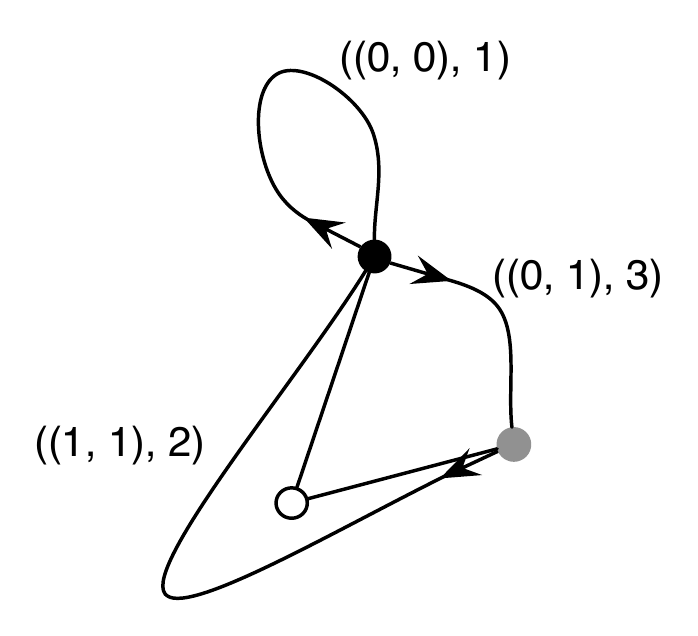}}
\caption{A $\Gamma_4$-crystallographic framework:
(a) A piece of an infinite crystallographic framework with $\Gamma_4$ symmetry.  The group
$\Gamma_4$ is generated by an order $4$ rotation and translations.  The fundamental domain of the
$\Gamma_4$-action on $\R^2$ is shown as a dashed box.
(b) The associated colored graph capturing the underlying combinatorics.  The color coding
conventions are as in \figref{gam2graphtoperiodic}.
}
\label{fig:gam4graphtoperiodic}
\end{figure}

\subsection{Algebraic setup and combinatorial model}
A $\Gamma$-crystallographic framework is given by the data $\Gal$.  The infinite
graph $\tilde{G}$ encodes the combinatorial structure of the bars.  The crystallographic
group $\Gamma$, along with the free $\Gamma$-action $\varphi$ on $\tilde{G}$ by
automorphisms (i.e., $\varphi : \Gamma\to\operatorname{Aut}(\tilde{G})$ is a
representation) determines the framework's symmetry; for convenience, we define the
notation $\gamma(i) := \varphi(\gamma)(i)$ for $\gamma\in \Gamma$ and $i\in \Vtilde$.
The rest of the framework's geometric data is given by the vector $\elltilde$, which is an
assignment of a positive length to each edge $ij\in \Etilde$.

We assume that $\Gtilde$
has finite quotient $G = \Gtilde/\Gamma$ with $n$ vertices and $m$ edges.
To keep the terminology in this framework manageable,
we will refer simply to \emph{frameworks} when the context is clear, with the understanding that the
frameworks appearing in the paper are crystallographic.

A \emph{realization} $\Gpa$ of the abstract framework $\Gal$ is defined to be an assignment
$\vec p=\left(\vec p_{i}\right)_{i\in\Vtilde}$ of points to the vertices of $\Gtilde$ and
a representation $\Phi$ from $\Gamma$ to a Euclidean isometry group, such that
\begin{align}
||\vec p_i - \vec p_j||  =   \elltilde_{ij} & \text{\qquad for all edges $ij\in \Etilde$} \label{lengths} \\
\Phi(\gamma)\cdot \vec p_i  =  \vec p_{\gamma(i)} & \text{\qquad for all group elements $\gamma\in \Gamma$ and vertices $i\in\Vtilde$} \label{equivariant}
\end{align}
The condition \eqref{lengths}, which appears in the theory of finite frameworks, says that a realization
respects the given edge lengths.  Equation \eqref{equivariant} says that, if we hold $\Phi$ fixed, regarded as a map
$\vec p : \Vtilde\to \R^2$, $\vec p$ is equivariant with respect to the $\Gamma$-actions on $\tilde{G}$ and $\R^2$.
However, $\Phi$ is, in general, \emph{not} fixed.
This is a very important feature of the model: the motions  available to the framework include those that deform
the representation $\Phi$ of $\Gamma$, provided this happens in a way compatible with the $\Gamma$-action
$\varphi$.

The \emph{realization space} $\mathcal{R}\Gal$  (shortly $\mathcal{R}$)
of an abstract framework is defined as the set of its realizations.  The \emph{configuration space}
$\mathcal{C}$ is defined to be the quotient of $\mathcal{R}$ by Euclidean isometries.  A realization  $\tilde{G}(\vec p, \Phi)$
is \emph{rigid} if it is isolated in $\mathcal{C}$ and otherwise \emph{flexible}.
(See \secref{continuous} for a detailed treatment of these spaces.)

As the combinatorial model for crystallographic frameworks it will be more convenient to use colored graphs.
A \emph{colored graph} $(G,\bgamma)$ is a finite, directed graph $G$, with an assignment
$\bgamma = (\gamma_{ij})_{ij\in E(G)}$ of an element of a group $\Gamma$ to each edge.

A straightforward specialization of covering space theory, described in \secref{colored-graphs},
associates $(\tilde{G},\varphi)$ with a colored graph $(G,\bgamma)$: $G$ is the quotient of $\tilde{G}$
by $\Gamma$, and the colors encode the covering map $\tilde{G} \to G$ via a map
$\rho : \pi_1(G,b) \to \Gamma$.

\subsection{Main theorem}\seclab{mainstatement}
Our main result is the following ``Maxwell-Laman-type''
theorem for  crystallographic frameworks where the symmetry group is
generated by translations and a finite order rotation.
The ``\emph{$\Gamma$-colored-Laman graphs}'' appearing in the statement are defined in \secref{gamma-laman};
genericity is defined in detail in \secref{infinitesimal}, but the term is used in the standard sense of
algebraic geometry: generic frameworks are the (open, dense) complement of a proper algebraic
subset of $\mathbb{R}^{m}$.
\begin{restatable}{theorem}{mainthm} \theolab{main}%
Let $\Gamma$ be an orientation-preserving crystallographic group.
A generic crystallographic framework $(\tilde{G}, \varphi, \tilde{\bm{\ell}})$
with symmetry group $\Gamma$
is minimally rigid if and only if its colored quotient graph is $\Gamma$-colored-Laman.
\end{restatable}
Whether a colored graph is $\Gamma$-colored-Laman can be checked in polynomial time
by combinatorial algorithms based on Edmonds's augmenting path algorithm for Matroid Union
\cite{E65}.

\subsection{Infinitesimal rigidity and direction networks}
In order to prove the rigidity \theoref{main}, we will prove a
combinatorial characterization of \emph{generic infinitesimal rigidity},
which is a linearization of the problem.  Standard kinds of arguments, along the
lines of \cite{AR78}, imply that, generically, rigidity and infinitesimal rigidity
coincide.  We will study infinitesimal rigidity using crystallographic direction
networks.

A \emph{$\Gamma$-crystallographic direction network} $\Gad$ consists of an infinite graph $\Gtilde$
with a free $\Gamma$-action $\varphi$ on the edges and vertices, and an assignment of a \emph{direction}
$\tilde{\vec d}_{ij}\in \R^2\setminus\{0\}$ to each edge $ij\in \tilde{E}$.
We define a realization $\Gpa$ of $\Gad$ to be a mapping of $\Vtilde$ to a point set $\vec p$ and a
representation $\Phi$ of $\Gamma$ by Euclidean isometries such that
\begin{align}
\iprod{\vec p_i - \vec p_j}{\tilde{\vec d}_{ij}^\perp}  =  0 & \text{\qquad for all edges $ij\in \Etilde$} \label{directions}
\\
\Phi(\gamma)\cdot\vec p_i  =  \vec p_{\gamma(i)} & \text{\qquad for all group elements $\gamma\in \Gamma$ and vertices $i\in\Vtilde$}
\label{equivariantD}
\end{align}
Equation \eqref{directions} says that, in any realization, $\vec p_i - \vec p_j$ is a scalar multiple of $\tilde{\vec d}_{ij}$,
for each edge $ij\in \Etilde$; \eqref{equivariantD} gives the symmetry constraint.
Since setting all the $\vec p_i$ equal and $\Phi$ to be trivial produces a realization, the realization space is never empty.
For our purpose, though, such realizations are degenerate.  We define a realization of a crystallographic
direction network to be \emph{faithful} if none of the edges of $\Gtilde$ are realized with coincident endpoints.
Our second main result is an exact characterization of when a generic direction network
admits a faithful realization.
\begin{restatable}{theorem}{directionthm} \theolab{direction}%
Let $\Gamma$ be an orientation-preserving crystallographic group.  A
generic $\Gamma$-crystallographic direction network $\Gad$ has a unique, up to translation and scaling,
faithful realization if and only if its associated colored graph is $\Gamma$-colored-Laman.
\end{restatable}
\subsection{Roadmap and novelty}
The overall methodology is an adaptation of the \emph{direction network method} (cf. \cite{ST10} and
\cite[Section 4]{W88}) for proving rigidity characterizations in the plane.  The
reduction from infinitesimal rigidity to direction network realizability is, by now,
fairly standard.  Thus, most of the novelty lies in proving \theoref{direction}.  This
is done in three main steps: (i) the construction of a matroid on an orientation-preserving
crystallographic group (\secref{bigsec-groups}); (ii) an extension of the group matroid
to one on graphs that serves as a kind of ``generalized graphic matroid'' (\secref{bigsec-graphs});
(iii) a linear representation result
relating bases of the new combinatorial matroid to direction networks with only trivial ``collapsed''
realizations (\secref{bigsec-dn}).  The last step uses a new kind of geometric argument that is not a
straightforward reduction to the Matroid Union Theorem as in \cite{ST10}; the matroids constructed
here, to our knowledge, appear for the first time here (and in \cite{MT11}).

\subsection{History and related work}
This paper is the second in a sequence derived from the preprints \cite{MT11,MT12a},
and the material here has appeared, with the same proofs in \cite{MT11}.  The
first part is the submitted manuscript \cite{MT12fsi}.  The results here
are built on the theory we developed for studying \emph{periodic frameworks} in
\cite{MT13}, which contains a detailed discussion of motivations and
other work on periodic frameworks.

The general area of rigidity with symmetry has been somewhat active in the past few years.
For completeness, we review some work along
similar lines. A specialization of our \cite[Theorem A]{MT13} is
due to Ross \cite{R11}. Schulze \cite{S10a,S10b} and Schulze and Whiteley \cite{SW10} studied the
question of when ``incidental'' symmetry induces non-generic behaviors in finite frameworks, which is
a different setting than the forced symmetry we consider here and in \cite{MT13, MT12fsi}, however
one can interpret some of those results in the present setting.
Ross, Schulze, and Whiteley \cite{RSW10} have studied the
problem we do here, but they do not give any combinatorial characterizations.  Borcea and Streinu
\cite{BS11} have proposed a kind of ``doubly generic'' periodic rigidity, where the
combinatorial model does not include the colors on the quotient graph.

A recent preprint of Tanigawa \cite{T12}  proves a number
of parallel redrawing and body-bar rigidity characterizations in higher dimensions and
for a larger number of groups than considered here.  The method of \cite{T12} is, essentially,
to axiomatize the properties of the rank function of the matroid we construct in
\secref{group-matroid} and then follow a similar program, making use of a
new generalization of Matroid Union.

\subsection{Acknowledgements} We thank Igor Rivin for encouraging us to take on this project and
many productive discussions on the topic.  Our initial work on this topic was
part of a larger effort to understand the rigidity and flexibility of hypothetical zeolites,
which is supported by NSF CDI-I grant DMR 0835586 to Rivin and M. M. J. Treacy.  LT is funded by the
European Research Council under the European Union's Seventh Framework Programme (FP7/2007-2013) /
ERC grant agreement no 247029-SDModels.  JM is supported by the European Research Council under the
European Union’s Seventh Framework Programme (FP7/2007-2013) / ERC grant agreement no 226135.

\section{Groups and matroids}\seclab{bigsec-groups}
\subsection{Crystallographic group preliminaries}\seclab{crystal-prelim}
We first review some basic facts about orientation-preserving crystallographic groups.

\subsubsection{Facts about the Euclidean group}
The Euclidean isometry group $\Euc(d)$ in any dimension can be represented as the semidirect product $\R^d \rtimes O(d)$
where $O(d)$ is the orthogonal group acting on $\R^d$ in the standard way. The group operation is thus:
\[
(\vec v, r) \cdot (\vec v', r') = (\vec v + r \cdot \vec v', r r')
\]
The subgroup $\R^d < \Euc(d)$ is the translation subgroup, and the projection $\Euc(d) \to O(d)$
is the map that associates to an isometry $\psi$ its derivative at the origin $D\psi_0$.
The action of $(\vec v, r) \in \Euc(d)$ on a point $\vec p \in \R^d$ is
$(\vec v, r) \cdot \vec p = \vec v + r \cdot \vec p$.

Since our setting is $2$-dimensional, from now on, we are interested in $\Euc(2)$.  In the two dimensional case,
we have the following simple lemma, which we state without proof.
\begin{lemma} \lemlab{TorR}
Any nontrivial orientation-preserving isometry of the Euclidean plane is either a rotation around
a point or a translation.
\end{lemma}
Thus, when we refer to orientation-preserving elements of $\Euc(2)$ we call them simply \emph{``rotations''}
or \emph{``translations''}.  We denote the counterclockwise rotation around the origin through angle $2\pi/k$
by $R_k$.

\subsubsection{Crystallographic groups}
A {\it $2$-dimensional crystallographic group} $\Gamma$ is a group admitting a discrete cocompact
faithful representation $\Gamma \to \Euc(2)$.  We will denote by $\Phi$ such representations
of $\Gamma$.  In this paper, we are interested in the case where all the group elements are represented
by rotations and translations (i.e., we disallow reflections and glides).

The enumeration of the $2$-dimensional crystallographic groups is classical, and there are precisely five
orientation-preserving ones (see, e.g., \cite{CDHT01}).
The first group, which we denote by $\Gamma_1$, is $\Z^2$.  The rest
are all semidirect products of $\Z^2$ with a cyclic group.  Namely, for $k = 2, 3, 4, 6$, we have
$\Gamma_k = \Z^2 \rtimes \Z/k\Z$.  The action on $\Z^2$ by the generator
of $\Z/k\Z$ is given by the following table.

\begin{center}
\begin{tabular}{|c|c|c|c|c|}
\hline  $k$ & $2$ & $3$ & $4$ & $6$ \\
\hline matrix & $\JMat{-1 & 0 \\ 0 & -1}$  & $\JMat{ 0 & -1 \\ 1 & -1} $ & $\JMat{0 & -1 \\ 1 & 0} $
&  $\JMat{0 & -1 \\ 1 & 1} $ \\
\hline
\end{tabular}
\end{center}

\noindent
We define the $\Z^2$ subgroup of $\Gamma_k$ to be the {\em translation subgroup} of $\Gamma_k$ and denote it by $\Trans(\Gamma_k)$.
We denote $\gamma\in \Gamma_k$, $k=2,3,4,6$ as $\gamma=(t,r)$ with $t\in \Z^2$ and $r\in \Z/k\Z$.

\subsubsection{Remark on groups considered}
Since we are only interested in crystallographic groups of this form, the rest of the paper will consider $\Gamma_k$ only (and not
more general crystallographic groups).  Moreover,
we will treat only $k=2,3,4,6$ in what follows, because the case of $\Gamma_1$ is covered by \cite[Theorem A]{MT13}.
However, the theory and proof methods presented here specialize to $\Gamma_1$.

\subsubsection{Finitely generated subgroups} If $\gamma_1,\ldots, \gamma_t$ are element of $\Gamma_k$, we denote the subgroup
generated by the $\gamma_i$ as $\langle \gamma_1,\ldots, \gamma_t\rangle$.   If $\Gamma^1, \ldots, \Gamma^t$
are a sequence of finitely generated subgroups then $\langle \Gamma^1, \Gamma^2,\ldots, \Gamma^t \rangle$
denotes the subgroup generated by the union of some choice of generators for each $\Gamma^i$. We will sometimes abuse
notation and consider groups generated together by some elements and some subgroups, e.g.
$\langle \gamma_1, \gamma_2, \Gamma^1, \Gamma^2, \Gamma^3 \rangle$.

\subsection{Representation space}\seclab{repspace}

$\Gamma$-crystallographic frameworks and direction networks are required to be symmetric
with respect to the group $\Gamma$.  However, the representation is allowed to flex.  In this
section, we formalize this flexing.

\subsubsection{The representation space}
Let $\Gamma$ be a crystallographic group.  We define the \emph{representation space} $\Rep(\Gamma)$ of $\Gamma$ to be
\[
\Rep(\Gamma) = \{ \Phi : \Gamma \to \R^2 \rtimes O(2) \; | \; \Phi \text{ is a discrete faithful representation} \}
\]

\subsubsection{Motions in representation space}
For our purposes, a $1$-parameter family of representations is a continuous motion if it is pointwise continuous.
More precisely, identify $\Euc(2) \cong \R^2 \times O(2)$ as topological spaces.  Suppose
$\Phi_t: \Gamma \to \Euc(2)$ is a family of representations defined for $t \in (-\epsilon, \epsilon)$ for some $\epsilon > 0$.
Then, $\Phi_t$ is a continuous motion through $\Phi_0$
if $\Phi_t(\gamma)$ is a continuous path in $\Euc(2)$ for all $\gamma \in \Gamma$.

\subsubsection{Generators for $\Gamma_k$}
To describe the representation space, we need a description of the generating sets for each of the $\Gamma_k$, which follows
from their descriptions as semi-direct products of $\Z^2\rtimes \left(\Z/k\Z\right)$.
\begin{lemma}\lemlab{gensets}
The following are generating sets for each of the $\Gamma_k$:
\begin{itemize}
\item $\Gamma_2$ is generated by the set $\{((1,0),0), ((0,1),0), ((0,0), 1)\}$.
\item $\Gamma_k$ is generated by the set $\{((1,0),0), ((0,0), 1)\}$ for $k=3,4,6$.
\end{itemize}
\end{lemma}
For convenience, we set the notation $r_k=((0,0),1)$, $t_1=((1,0),0)$, and $t_2=((0,1),0)$.

\subsubsection{Coordinates for representations}
We now show how to give convenient coordinates for the representation space
for each $\Gamma_k$ for $k=2,3,4,6$; by the classification of
$2$-dimensional crystallographic groups, these are the only cases we need to check.
This next lemma follows readily from Bieberbach's Theorems \cite{B11,B12} and \lemref{gensets}, but we give a proof in
\secref{repspace-proof} for completeness.
\begin{lemma}\lemlab{repspace}
The representation spaces of each of the $\Gamma_k$ can be given coordinates as follows:
\begin{itemize}
\item $\Rep(\Gamma_2)\cong \{\vec v_1,\vec v_2, \vec w \in \R^2 : \text{$\vec v_1$ and $\vec v_2$ are linearly independent}\}$
\item $\Rep(\Gamma_k) \cong \{ \vec v_1, \vec w, \varepsilon \; | \; \vec v_1 \neq 0, \varepsilon = \pm 1, \vec v_1, \vec w \in \R^2 \}$ for $k=3,4,6$
\end{itemize}
\end{lemma}
The vectors specify the ``$\R^2$-part'' of the image of a generator in $\Euc(2) \cong \R^2 \rtimes O(2)$.
Specifically, $\vec v_i, \vec w$ are the coordinates for $\Phi$ precisely when $\Phi(t_i) = (\vec v_i, \Id)$
and $\Phi(r_k) = (\vec w, R_k)$. The vector $\vec w$ determines the rotation center, but is \emph{not} the rotation center itself.
(In fact, the rotation center is $(I - R_k)^{-1}(\vec w)$.)

\subsubsection{Coordinates for finite-order rotations}
The following lemma characterizes an order $k$ rotation in terms of the semidirect product $\R^2 \rtimes O(2)$.
\begin{lemma}\lemlab{orderk}
Let $\psi$ be an orientation-preserving element of $\Euc(2)$.  Then $\psi$ has order
$k=2,3,4,6$ if and only if it is of the form $(\vec w,R_k^{\pm 1})$, where $R_k$ is the order $k$
counterclockwise rotation through angle $2\pi/k$ and $\vec w\in \R^2$.
\end{lemma}
\begin{proof}
By \lemref{TorR}, $\psi$ is a rotation or a translation, and translations clearly have the form $(\vec w, \Id)$. Thus,
$\psi$ is a rotation if and only if it has the form $(\vec w, R)$ for some nontrivial rotation $R$. If $\psi^k$ has the form
$(\vec w', \Id)$, then $\vec w'$ is necessarily zero as no power of a rotation is a translation. Hence, the order of
$(\vec w, R)$ is precisely that of $R$ and the rest of the theorem follows easily.
\end{proof}

\subsubsection{Proof of \lemref{repspace}}\seclab{repspace-proof}
We let $\Phi \in \Rep(\Gamma_k)$ be a discrete, faithful representation.  Thus $\Phi$ is determined by the images of the generators,
so \lemref{gensets} tells us we need only to check $t_1$, $t_2$, and $r_k$.

The generators $t_i$ must always be mapped to translations: since they are infinite order and $\Phi$ is faithful,
the only other possibility is an infinite order rotation.  This would contradict $\Phi$ being discrete.  Thus:
\begin{itemize}
\item For $k= 2$, the elements $t_1$ and $t_2$ are mapped to translations $(\vec v_1, \Id)$ and $(\vec v_2, \Id)$.
\item For $k=3,4,6$, the element $t_1$ is mapped to a translation $(\vec v_1, \Id)$.
\end{itemize}
Moreover, faithfulness and discreteness force:
\begin{itemize}
\item All the images $\vec v_i$ to be non-zero.
\item The images $\vec v_1$ and $\vec v_2$ to be linearly independent for $k=1,2$.
\end{itemize}
By \lemref{orderk} we must have
$\Phi(r_k) = (\vec w, R_k^{\varepsilon})$ for some $w \in \R^2$ and $\varepsilon \in \{-1, 1\}$.  Since $R_2$ is order $2$,
we have $\Phi(r_2) = (\vec w, R_2) = (\vec w,R^{-1}_2)$, and so $\varepsilon$ is unnecessary
for $\Gamma_2$.

In the other direction, given the data described in the statement of the lemma, we simply define
$\Phi(t_i)$ and $\Phi(r_k)$ as above. When $k = 3, 4, 6$, we set $\Phi(t_2) = (R_k^{\varepsilon} \vec v_1, \Id)$.
For arbitrary elements of $\Gamma$, we define $\Phi( (m_1, m_2), m_3) = \Phi(t_1)^{m_1} \Phi(t_2)^{m_2} \Phi(r_k)^{m_3}$.
It is straightforward to check $\Phi$ as defined is a homomorphism, and that it is discrete and faithful. \eop

\subsubsection{Degenerate representations}
When we are dealing with  ``collapsed realizations'' of direction networks in \secref{bigsec-dn}, we
will need to work with certain degenerate representations of $\Gamma_k$.  The space
$$\overline{\Rep}(\Gamma_k)$$
is defined to be representations of $\Gamma_k$ where we allow the $\vec v_i$ to be any vectors.
Topologically this is the closure of  $\Rep(\Gamma_k)$ in the space of all (not necessarily discrete or faithful)
representations $\Gamma_k \to \Euc(2)$.

\subsubsection{Rotations and translations in crystallographic groups}
As we have defined them, $2$-dimensional crystallographic groups are abstract groups admitting a discrete
faithful representation to $\Euc(2)$.  However, as we
saw in the proof of \lemref{repspace}, all group elements in $\Trans(\Gamma_k)$ must be mapped to translations,
and all group elements outside $\Trans(\Gamma_k)$ must be mapped to rotations.  Consequently, we will henceforth call
elements of $\Trans(\Gamma_k)$ \emph{``translations''} and elements outside of $\Trans(\Gamma_k)$ \emph{``rotations''}
(even though technically they are elements of $\Gamma_k$, not $\Euc(2)$).

\subsection{Subgroup structure}\seclab{subgroups}
This short section contains some useful structural lemmas about subgroups of $\Gamma_k$.

\subsubsection{The translation subgroup}
For a subgroup $\Gamma' < \Gamma_k$, we define its \emph{translation subgroup} $\Trans(\Gamma')$ to be $\Gamma' \cap \Trans(\Gamma_k)$.
(Recall that $\Trans(\Gamma_k)$ is the subgroup $\Z^2$ coming from the semidirect product decomposition of $\Gamma_k$.)

\subsubsection{Facts about subgroups}
With all the definitions in place, we state several lemmas about subgroups of $\Gamma_k$ that we need later.
\begin{lemma} \label{lemma:subgrpgen}
Let $\Gamma' < \Gamma_k$ be a subgroup of $\Gamma_k$, and
suppose $\Gamma' \neq \Trans(\Gamma')$.  Then $\Gamma'$ is generated by one rotation and $\Trans(\Gamma')$.
\end{lemma}
\begin{proof}
We need only observe that $\Gamma_k/\Trans(\Gamma_k)$ is finite cyclic
and contains $\Gamma'_k/\Trans(\Gamma'_k)$ as a subgroup.
\end{proof}

This next lemma is straightforward, but useful.  We omit the proof.
\begin{lemma} \label{lemma:rotscomm}
Let $r_1, r_2 \in \Gamma_k$ be rotations.  Then $\langle r_1, r_2 \rangle$ is a finite cyclic subgroup consisting of
rotations if and only if some nontrivial powers $r_1^p$ and $r_2^q$ commute.
\end{lemma}

\begin{lemma} \label{lemma:Gam2nice}
Let $r' \in \Gamma_2$ be a rotation and $\Gamma' < \Trans(\Gamma_2)$ a subgroup of the translation
subgroup of $\Gamma_2$.
Then $\Trans( \langle r', \Gamma' \rangle) = \Gamma'$; i.e., after adding the rotation $r'$,
the translation subgroup of the group generated by $r'$ and $\Gamma'$ is again $\Gamma'$.
\end{lemma}
\begin{proof}
All translation subgroups of $\Gamma_2$ are normal, and so the set
$\{ gh \;\; | \;\; g \in \left\{r', \Id\right\} \;\; h \in \Gamma'\}$ is a subgroup and is equal to $\langle r', \Gamma' \rangle$.
Clearly, the only translations are those elements of $\Gamma'$.
\end{proof}

\subsection{The restricted representation space and its dimension}\seclab{radical}
To define our degree of freedom heuristics in \secref{bigsec-graphs},
we need to understand how representations of $\Gamma_k$ restrict to subgroups $\Gamma' < \Gamma_k$, or
equivalently, which representations of $\Gamma'$ extend to $\Gamma_k$.  For $\Gamma'<\Gamma_k$,
the \emph{restricted representation space} of $\Gamma'$ is the image of the restriction map
from $\Rep(\Gamma_k)$ to $\Rep(\Gamma')$, i.e.,
\[
\Rep_{\Gamma_k}(\Gamma') = \{ \Phi: \Gamma' \to \Euc(2) \; | \; \Phi
\text{ extends to a discrete faithful representation of $\Gamma_k$} \}
\]
We define the notation $\rep_{\Gamma_k}(\Gamma') := \dim \Rep_{\Gamma_k}(\Gamma')$, since the dimension of $\Rep_{\Gamma_k}(\Gamma')$ is
an important quantity in what follows. We also define the following invariant, which is essential for defining our combinatorial matroids:
\[
T(\Gamma') := \left\{ \begin{array}{rl} 0 & \text{ if $\Gamma'$ has a rotation} \\ 2 & \text{ if $\Gamma'$ has no rotations} \end{array} \right.
\]
Equivalently, we may define $T(\Gamma')$ as the dimension of the space of translations commuting
with $\Phi(\Gamma')$ for any $\Phi \in \Rep(\Gamma_k)$. In \secref{crystal-collapse}, we will show that $T(\Gamma')$
is the dimension of the space of collapsed solutions of a direction network for a connected graph
associated with the subgroup $\Gamma'$

We now develop some properties of $\rep_{\Gamma_k}(\cdot)$ and
how it changes as new generators are added to a finitely generated subgroup.
These will be important for counting the degrees of freedom
in direction networks (\secref{bigsec-dn}).

\subsubsection{Translation subgroups}
For translation subgroups $\Gamma'< \Gamma_k$, we are interested in the dimension of
$\Rep_{\Gamma_k}(\Gamma')$.  The following lemma gives a characterization for translation subgroups
in terms of the rank of $\Gamma'$.
\begin{lemma} \label{lemma:trrep}
Let $\Gamma' < \Gamma_k$ be a {\em nontrivial} subgroup of translations.
\begin{itemize}
\item If $k = 3, 4, 6$, then $\rep_{\Gamma_k}(\Gamma') = 2$.
\item  If $k = 2$, 	then $\rep_{\Gamma_k}(\Gamma') = 2 \cdot r$, where $r$ is the
rank of $\Gamma'$.
\end{itemize}
In particular, $\rep_{\Gamma_k}(\Gamma')$ is even.
\end{lemma}
\begin{proof}
Suppose $k = 3, 4$ or $6$.  By Lemma \ref{lemma:repspace}, the space of representations of $\Gamma_k$
is $4$-dimensional and is uniquely determined by the
parameters $\vec v_1, \vec  w$ and the sign $\varepsilon$.  The group $\Trans(\Gamma_k) \cong \Z^2$ is generated by $t_1$
and $r_k t_1 r_k^{-1}$, and so any $\gamma \in \Trans(\Gamma_k)$ can be written uniquely as
$t_1^{m_1} r_k t_1^{m_2} r_k^{-1}$ for integers $m_1, m_2$.
Thus, since $\Phi(\gamma)$ is a translation,
\begin{align*}
\Phi(\gamma) & =   \Phi(t_1)^{m_1} \Phi(r_k) \Phi(t_1)^{m_2} \Phi(r_k)^{-1} \\
& = \left(m_1 \vec v_1, \Id\right)  \left(\vec w, R_k^{\varepsilon}\right) \left(m_2 \vec v_1, \Id \right)\left(\vec w, R_k^{-\varepsilon} \right) \\
& = \left(m_1 \vec v_1 + m_2 R_k^\varepsilon \vec v_1, \Id\right)
\end{align*}
The computation shows that the restriction of $\Phi$ to $\Trans(\Gamma_k)$ is independent of
the parameter $\vec w$. Moreover two representations with the same parameter $\epsilon$ restrict to the
same representation of $\Trans(\Gamma_k)$ precisely when the $\vec v_1$ parameters are equivalent,
and $\vec v_1$ is completely determined by $\Phi(\gamma)$.

Suppose $k = 2$.  In this case by the proof of Lemma \ref{lemma:repspace}, any discrete faithful
representation $\Trans(\Gamma_k) \to \Euc(2)$ extends to a discrete faithful representation of $\Gamma_k$.
Since $\Trans(\Gamma_k) \cong \Z^2$, any discrete faithful representation of its subgroups to $\R^2$ %
extends to
$\Trans(\Gamma_k)$ and hence to $\Gamma_k$.  Hence $\rep_{\Gamma_k}(\Gamma')$ is equal to the dimension of
representations $\Gamma' \to \R^2$ which is twice the rank of $\Gamma'$.
\end{proof}

\subsubsection{The $r$-closure of a subgroup}
In \secref{group-matroid}, we will introduce a matroid on the elements of a crystallographic group.  To prove
the required properties, we need to know how the translation subgroup
$\Trans(\cdot)$ changes as generators are added to a
subgroup of $\Gamma_k$. %
We define the \emph{$r$-closure, $\Rad(\Gamma')$, of $\Gamma'$} to be the largest subgroup containing $\Gamma'$ such that
\begin{eqnarray}
\rep_{\Gamma_k}(\Trans(\Gamma')) = \rep_{\Gamma_k}(\Trans(\Rad(\Gamma'))) & \text{and} & T(\Gamma') = T(\Rad(\Gamma'))
\end{eqnarray}
The letter $r$ in this terminology refers to the rank function $r$ defined in \secref{group-matroid}, and the $r$-closure
is defined such that $\Rad(\Gamma')$ is the largest subgroup containing $\Gamma'$ with $r(\Rad(\Gamma')) = r(\Gamma')$.
The properties of the $r$-closure are needed to study the matroid defined by the closely related rank function $g_1$ (also
in \secref{group-matroid}),  which is a building block for the definition of $\Gamma$-Laman graphs in
\secref{gamma-laman}. Since there will be no confusion, we will henceforth drop the $r$ and simply refer to closures of subgroups.

\subsubsection{Properties of the closure}
This next sequence of lemmas enumerates the properties of the closure that we will
use in the sequel.

\begin{lemma}\lemlab{rad-welldefined}
Let $\Gamma'<\Gamma_k$ be a subgroup of $\Gamma_k$.  Then, the
closure $\Rad(\Gamma')$ is well-defined. Specifically for $k = 2$,
\begin{itemize}
\item If $\Gamma'$ is a translation subgroup, then $\Rad(\Gamma')$ is the subgroup of translations with
a non-trivial power in $\Gamma'$.
\item If $\Gamma'$ has translations and rotations, then $\Rad(\Gamma') = \langle r', \Rad(\Trans(\Gamma'))\rangle$
for any rotation $r' \in \Gamma'$.
\end{itemize}
For $k=3,4,6$, there are four possibilities for the closure:
\begin{itemize}
\item If $\Gamma'$ is trivial, then the closure is trivial.
\item If $\Gamma'$ is cyclic, then the closure is a cyclic subgroup of order $k$.
\item If $\Gamma'$ is a nontrivial translation subgroup, then the closure is the translation subgroup of $\Gamma_k$.
\item If $\Gamma'$ has translations and rotations, then the closure is all of $\Gamma_k$.
\end{itemize}
\end{lemma}
\begin{proof}
First let $k=2$.  There are two cases.
If $\Gamma'$ contains only translations, we set
\[
\Rad(\Gamma') = \{t\in \Trans(\Gamma_2) : \text{$t^i\in \Gamma'$ for some power $i$ of $t$} \}
\]
Any subgroup $\Gamma''<\Gamma_2$ containing $\Gamma'$ with $T(\Gamma')=T(\Gamma'')$ and
$\rep_{\Gamma_k}(\Gamma')=\rep_{\Gamma_k}(\Gamma'')$ must be a	translation group of the same rank as $\Gamma'$ and
$\Rad(\Gamma')$ is the largest such subgroup.

Otherwise, $\Gamma'$ contains a rotation $r'$.  In this case, we set
\[
\Rad(\Gamma') = \langle r', \Rad(\Trans(\Gamma'))\rangle
\]
By \lemref{Gam2nice}, for $\Rad(\Gamma')$ defined this way, the translation subgroup
$\Trans(\Rad(\Gamma'))$ is just $\Rad(\Trans(\Gamma'))$
which by the previous paragraph is the largest translation subgroup containing $\Trans(\Gamma')$ and having the same rank.
Suppose $\Gamma'' < \Gamma_2$ contains $\Gamma'$ and satisfies
$\rep_{\Gamma_k}(\Trans(\Gamma'')) = \rep_{\Gamma_k}(\Trans(\Gamma'))$ and $T(\Gamma'') = T(\Gamma')$.
Then, $\Gamma'' = \langle r', \Trans(\Gamma'') \rangle$ and $\Trans(\Gamma')$ and $\Trans(\Gamma'')$ have the
same rank. This implies that $\Trans(\Gamma'') < \Rad(\Trans(\Gamma'))$ and thus $\Gamma'' < \Rad(\Gamma')$.

Now we suppose that $k=3,4,6$.  There are four possibilities for $\Gamma'$:
\begin{itemize}
\item If $\Gamma'$ is trivial, then we define $\Rad(\Gamma') = \Gamma'$.
\item If $\Gamma'$ is a cyclic group of rotations, then \lemref{rotscomm} guarantees that there is a unique
largest cyclic subgroup containing it, and we define this to be $\Rad(\Gamma')$. Any larger group
will have a different $\rep_{\Gamma_k}$ value.
\item If $\Gamma'$ has only translations, then we define $\Rad(\Gamma') = \Trans(\Gamma_k)$. From
Lemma \ref{lemma:trrep} it follows that $\rep_{\Gamma_k}(\Trans(\Gamma')) = \rep_{\Gamma_k}(\Trans(\Rad(\Gamma')))$.
Any larger subgroup will have a different $T(\cdot)$ value.
\item If $\Gamma'$ has translations and rotations, then it has the same $\rep_{\Gamma_k}(\Trans(\cdot))$ and
$T(\cdot)$ values as $\Gamma_k$, so $\Rad(\Gamma') = \Gamma_k$.
\end{itemize}
\end{proof}

\begin{lemma}\lemlab{rad-monotone}
Let $\Gamma'<\Gamma_k$ be a finitely-generated subgroup of $\Gamma_k$,
and let $\Gamma''< \Gamma'$ be a subgroup of $\Gamma'$.  Then $\Rad(\Gamma'') < \Rad(\Gamma')$.
\end{lemma}
\begin{proof}
Pick a generating set of $\Gamma''$ that extends to a generating set of $\Gamma'$.
Analyzing the cases in \lemref{rad-welldefined} shows that
the closure cannot become smaller after adding generators.
\end{proof}

\begin{lemma}\lemlab{rad-conj}
Let $\Gamma'<\Gamma_k$ be a translation subgroup of $\Gamma_k$, and let $\gamma\in \Gamma_k$.  Then
$\Rad(\gamma\Gamma'\gamma^{-1}) = \Rad(\Gamma')$; i.e., the closure of translation subgroups is
fixed under conjugation.
\end{lemma}
\begin{proof}
For $k=2$ this follows from the fact that all translation subgroups are normal.  For $k=3,4,6$ it
is immediate from \lemref{rad-welldefined}.
\end{proof}

\begin{lemma}\lemlab{rad-push-trans}
Let $\Gamma'<\Gamma_k$ be a subgroup of $\Gamma_k$, and let $\Gamma''<\Trans(\Gamma_k)$ be a translation subgroup
of $\Gamma_k$.  Then $\Rad( \langle \Trans(\Gamma'), \Gamma''\rangle) = \Rad(\Trans(\langle \Gamma', \Gamma'' \rangle))$.
\end{lemma}
\begin{proof}
The proof is in cases based on $k$.  For $k=3,4,6$, \lemref{rad-welldefined} implies that
either $\Gamma''$ is trivial or both sides of the desired equation are
$\Trans(\Gamma_k)$.  In either case, the lemma follows at once.

Now suppose that $k=2$.  If $\Gamma'$ is a translation subgroup, then the lemma follows immediately.  Otherwise,
we know that $\Gamma'$ is generated by a rotation $r'$ and the translation subgroup $\Trans(\Gamma')$.  Applying
\lemref{Gam2nice}, we see that
\[
\Trans(\langle\Gamma', \Gamma'' \rangle) =
\Trans(\langle r', \Trans(\Gamma'), \Gamma'' \rangle) = \langle \Trans(\Gamma'), \Gamma''\rangle
\]
from which the lemma follows.
\end{proof}

\subsubsection{The quantity $\rep_{\Gamma_k}(\Trans(\Gamma')) - T(\Gamma')$}
The following statement plays a key role in the matroidal construction of \secref{group-matroid}.

\begin{prop}\proplab{rad-rep-minus-T}
Let $\Gamma'<\Gamma_k$ be a subgroup of $\Gamma_k$, and let $\gamma\in \Gamma_k$ be an element of $\Gamma_k$.  Then,
\[
\rep_{\Gamma_k}(\Trans(\langle \Gamma', \gamma \rangle)) -
T(\langle \Gamma', \gamma \rangle) - (\rep_{\Gamma_k}(\Trans(\Gamma')) - T(\Gamma'))
=
\left\{ \begin{array}{rl} 2 & \text{ if } \gamma \notin \Rad(\Gamma') \\
0 & \text{ otherwise} \end{array} \right.
\]
i.e., the quantity $\rep_{\Gamma_k}(\Lambda(\cdot))-T(\cdot)$ increases by two after adding
$\gamma$ to $\Gamma'$ if and only if $\gamma\notin\Rad(\Gamma')$ and otherwise the
increase is zero.
\end{prop}
\begin{proof}
If $\gamma \in \Rad(\Gamma')$, this follows at once from the definition, since the
quantity $\rep_{\Gamma_k}(\Trans(\Gamma'))-T(\Gamma')$ depends only on the closure.

Now suppose that $\gamma \notin \Rad(\Gamma')$.  Since the closure is defined in terms of
$\rep_{\Gamma_k}(\Lambda(\cdot))$ and $T(\cdot)$, \lemref{rad-monotone} implies that
at least one of $\rep_{\Gamma_k}(\Lambda(\cdot))$ or $-T(\cdot)$ increases. It is easy to see from the definition
that either type of increase is by at least $2$.  We will show that the increase is at most $2$, from which
the lemma follows.  The rest of the proof is in three cases, depending on $k$.

Now we let $k=3, 4, 6$.  The only way for the increase to be larger than $2$ is for
$\Gamma'$ to be trivial and $\Rad(\langle \gamma \rangle) = \Gamma_k$.  This is
impossible given the description from \lemref{rad-welldefined}.

To finish, we address the case $k = 2$.
Suppose $\gamma$ is a translation.  Then $T(\langle \gamma, \Gamma' \rangle) = T(\langle \Gamma'\rangle)$, since
adding $\gamma$ as a generator does not give us a new rotation if one was not already present in $\Gamma'$.
Lemmas \ref{lemma:subgrpgen} and \ref{lemma:Gam2nice} imply that
$\Trans(\langle \gamma, \Gamma' \rangle) = \langle \gamma, \Trans(\Gamma') \rangle$.
Hence, the rank of the translation subgroup increases by at most $1$, and so, by \lemref{trrep},
$\rep_{\Gamma_k}(\cdot)$ increases by at most $2$.

Now suppose that $\gamma$ is a rotation.  If $\Gamma'$ has no rotations,
then Lemma \ref{lemma:Gam2nice} implies $\Trans(\langle \gamma, \Gamma' \rangle) = \Gamma'$,
and so $T(\cdot)$ decreases and $\rep_{\Gamma_2}(\cdot)$ is unchanged.  If
$\Gamma'$ has rotations, then $\Gamma' = \langle r', \Trans(\Gamma')\rangle$ for some rotation $r'\in \Gamma'$.
Since $k=2$, the product $r' \gamma$ is a translation and so
\[\Trans(\langle \gamma, \Gamma' \rangle) = \Trans(\langle \gamma, r', \Trans(\Gamma') \rangle) =
\Trans(\langle r', r' \gamma, \Trans(\Gamma') \rangle) = \langle r' \gamma, \Trans(\Gamma') \rangle\]
Thus, in this case, the the number of generators of the translation subgroup increases by at most one and
$T(\cdot)$ is unchanged.  By \lemref{trrep}, the proof is complete.
\end{proof}

\subsection{Teichm\"uller space and the centralizer}\seclab{teich}
The representation spaces defined in the previous two sections are closely related to the degrees
of freedom in the crystallographic direction networks we study in the sequel.  In this section, we
discuss the \emph{Teichmüller space} and \emph{centralizer}, which play the same role for frameworks.

\subsubsection{Teichmüller space}
The \emph{Teichm\"uller space} of $\Gamma_k$ is defined to be
the space of discrete faithful representations, modulo conjugation by
$\Euc(2)$; i.e. $\Teich(\Gamma_k) = \Rep(\Gamma_k)/\Euc(2)$.  For a subgroup
$\Gamma' < \Gamma_k$, we define its \emph{restricted Teichm\"uller space} to be
\[
\Teich_{\Gamma_k}(\Gamma') = \Rep_{\Gamma_k}(\Gamma')/\Euc(2)
\]
Correspondingly, we define $\teich_{\Gamma_k}(\Gamma') = \dim(\Teich_{\Gamma_k}(\Gamma'))$.

\subsubsection{The centralizer}
For a subgroup $\Gamma' \leq \Gamma_k$ and a discrete faithful representation $\Phi: \Gamma_k \to \Euc(2)$,
the \emph{centralizer of $\Phi(\Gamma')$} which we denote $\Cent_{\Euc(2)}(\Phi(\Gamma'))$ is the set of
elements commuting with all elements in $\Phi(\Gamma')$.
We define $\cent(\Gamma')$ to be the dimension of the centralizer
$\Cent_{\Euc(2)}(\Phi(\Gamma'))$.  The quantity $\cent(\Gamma')$ is independent of $\Phi$,
and we can compute it.  Since we do not depend on \lemref{centtable} or \propref{reptoteich} for any of our main results,
we skip the proofs in the interest of space.
\begin{lemma} \lemlab{centtable}
Let notation be as above.  The quantity $\cent(\Gamma')$
is independent of the representation $\Phi$.
Furthermore, $\cent(\Gamma') \geq T(\Gamma')$, and in particular,
\[ \cent(\Gamma') =
\begin{cases}
0 & \text{ if $\Gamma'$ contains rotations and translations} \\
1 & \text{ if  $\Gamma'$ contains only rotations} \\
2 & \text{ if  $\Gamma'$ contains only translations} \\
3 & \text{ if  $\Gamma'$ is trivial}
\end{cases}
\]
\end{lemma}
As a corollary, we get the following proposition relating $\rep_{\Gamma_k}(\cdot)$ and
$T(\cdot)$ to $\teich_{\Gamma_k}(\cdot)$ and $\cent(\cdot)$.
\begin{prop} \proplab{reptoteich}
Let $\Gamma' < \Gamma_k$.  Then:
\begin{itemize}
\item[\textbf{(A)}] If $\Gamma'$ contains a translation, then $T(\Gamma') = \cent(\Gamma')$.  Otherwise,
$T(\Gamma') = \cent(\Gamma') - 1$.
\item[\textbf{(B)}] If $\Gamma'$ is a non-trivial translation subgroup, then
$\teich_{\Gamma_k}(\Gamma') = \rep_{\Gamma_k}(\Gamma') -1$.
\item[\textbf{(C)}] If $\Gamma'$ is trivial, then $\teich_\Gamma(\Gamma') = \rep_\Gamma(\Gamma') = 0$.
\item[\textbf{(D)}] For any $\Gamma'<\Gamma_k$, $\rep_{\Gamma_k}(\Trans(\Gamma')) - T(\Gamma') =
\teich_{\Gamma_k}(\Trans(\Gamma')) - \cent(\Gamma')$.
\end{itemize}
\end{prop}

\subsection{A matroid on crystallographic groups}\seclab{group-matroid}
We now define and study a matroid $M_{\Gamma_k,n}$ for $k=2,3,4,6$.

\subsubsection{Preview of $\Gamma$-$(1,1)$ graphs and $M_{\Gamma_k,n}$}
In \secref{sparse}, we will relate
$M_{\Gamma_k,n}$ to ``$\Gamma$-$(1,1)$ graphs'', which are defined in \secref{gamma11-def}.  The
results here, roughly speaking, are the group theoretic part of the proof of \propref{gamma11}
in \secref{sparse}.

We briefly motivative the definitions given next. In general, $\Gamma$-$(1,1)$ graphs need
not be connected, and
each connected component has an associated finitely generated subgroup of $\Gamma_k$.  The ground
set of $M_{\Gamma_k, n}$ and the $A_i$ defined below capture this situation.  The operations of
conjugating and fusing, defined here in Sections \secrefX{Mkn-conj} and \secrefX{Mkn-fuse} will
be interpreted graph theoretically in \secref{sparse}.

\subsubsection{The ground set} For the definition of the ground set, we fix $\Gamma_k$ and a natural number $n\ge 1$.
The ground set $E_{\Gamma_k,n}$ is defined to be:
\[
E_{\Gamma_k,n} =
\left\{
(\gamma,i) : 1\le i\le n
\right\}
\]
In other words the ground set is $n$ labeled copies of $\Gamma_k$.

Let $A\subset E_{\Gamma_k,n}$.  We define some notation:
\begin{itemize}
\item $A_i = \{\gamma : (\gamma,i)\in A\}$; i.e., $A_i$ is the group elements from copy $i$ of $\Gamma_k$ in $A$.  Some of the
$A_i$ may be empty and $A_i$ can be a multi-set.  $A$ may equivalently be defined by the $A_i$.
\item $\Gamma_{A,i} = \langle \gamma : \gamma\in A_i\rangle$; i.e., the subgroup generated by the elements in $A_i$.
\item $\Trans(A) = \langle \Trans(\Gamma_{A,1}), \Trans(\Gamma_{A,2}), \dots, \Trans(\Gamma_{A,n}), \rangle$; the translation subgroup
generated by the translations in each of the $\Gamma_{A,i}$.
\item $c(A)$ is the number of $A_i$ that are not empty.
\end{itemize}

\subsubsection{The rank function}\seclab{g1}
We now define the function $g_1(A)$ for $A\subset E_{\Gamma_k,n}$ to be
\[
g_1(A) = n + \frac{1}{2}\rep_{\Gamma_k}(\Trans(A)) - \frac{1}{2}\sum_{i=1}^n T(\Gamma_{A,i})
\]

The meaning of the terms in $g_1(A)$ are as follows:
\begin{itemize}
\item The second term is a global adjustment for the representation space of the group generated by
the translations in each of the $\Gamma_{A,i}$.  We note that this is \emph{not} the same as
the translation group $\Trans(\langle\gamma : \gamma\in \cup_{i=1}^n A_i)$, which includes
translations arising as products of rotations in different $A_i$.
\item The quantity $n - \frac{1}{2} \sum_{i=1}^n T(\Gamma_{A, i}) = \sum_{i=1}^n (1 - \frac{1}{2} T(\Gamma_{A, i}))$
is a local adjustment based on whether $\Gamma_{A,i}$ contains a rotation:
each term in the latter sum is one if $\Gamma_{A,i}$ contains a rotation and otherwise it contributes
nothing.
\end{itemize}

\subsubsection{An analogy to uniform linear matroids}
To give some intuition about why the construction above might be matroidal, we observe that
\propref{rad-rep-minus-T}, interpreted in matroidal language gives us:
\begin{prop}\proplab{gammak-matroid}
Let $A$ be a finite subset of $\Gamma_k$ generating a subgroup $\Gamma_A$.  Then the function
\[
r(A) = \frac{1}{2}\left(\rep_{\Gamma_k}(\Lambda(\Gamma_A)) - T(\Gamma_A)\right)
\]
is the rank function of a matroid on the ground set $\Gamma_k$.
\end{prop}
The matroid in the conclusion of \propref{gammak-matroid} is analogous to a  linear matroid, with
$\Gamma_k$ playing the role of a vector space and $r$ the role of dimension of the linear span.
(And, in fact, for the group $\Z^2$, $r$ reduces simply to linear independence, as in \cite[Section 4]{MT13}.)
Since
the function $g_1$, defined above, builds on $r$, one might expect that it inherits a matroidal
structure.  We verify this next.

\subsubsection{$M_{\Gamma_k,n}$ is a matroid}
The following proposition is the main result of this section.
\begin{prop}\proplab{Mgammarank}
The function $g_1$ is the rank function of a matroid $M_{\Gamma_k,n}$.
\end{prop}
We note that although the ground set is infinite, since our matroids are finite rank,
all the facts for finite matroids which we cite apply here as well.

The proof depends on  Lemmas \lemrefX{reduce1} and \lemrefX{reduce2} below, so we defer it for the moment to \secref{Mgammarank-proof}.
The strategy is based on the observation that when $n=1$, the ground set is essentially $\Gamma_k$.  In this case,
submodularity and normalization of $g_1$ (the most difficult properties to establish) follow
immediately from \propref{rad-rep-minus-T}.  The motivation of Lemmas \ref{lemma:reduce1} and \ref{lemma:reduce2} is
to reduce, as much as possible, the proof of the general case to $n=1$.

\begin{lemma}\lemlab{reduce1}
Let $A \subset E_{\Gamma_k, n}$, and set $\Gamma_{A, \ell}' = \langle \Gamma_{A, \ell}, \Trans(A) \rangle$.
Then, for all $1 \leq \ell \leq n$,
\begin{itemize}
\item $\Rad(\Trans(A))  =  \Rad(\Trans(\Gamma_{A, \ell}'))$
\item $T(\Gamma_{A, \ell})   =  T(\Gamma_{A, \ell}') $
\end{itemize}
\end{lemma}
\begin{proof}
The first statement is immediate from \lemref{rad-push-trans}. The second statement follows from
the fact that $\Trans(A)$ is a translation subgroup of $\Gamma_k$, so $\Gamma_{A, \ell}'$ has a rotation
if and only if $\Gamma_{A,\ell}$ does.
\end{proof}

\begin{lemma}\lemlab{reduce2}
Let $A \subset E_{\Gamma_k, n}$, and set $\Gamma_{A, \ell}' = \langle \Gamma_{A, \ell}, \Trans(A) \rangle$.
If $B = A + (\gamma, \ell)$ and $\Gamma_{B, \ell}' = \langle \Gamma_{B, \ell}, \Trans(B) \rangle$,  Then,
\[
\Gamma_{B, \ell}' = \langle \gamma, \Gamma_{A, \ell}' \rangle
\]
\end{lemma}
\begin{proof}
First we observe that
\[
\Gamma_{B, \ell}' =
\langle \gamma, \Gamma_{A, \ell}, \Trans(B) \rangle \geq
\langle \gamma, \Gamma_{A, \ell}, \Trans(A) \rangle
\]
so to finish the proof we just have to show that
\[
\Trans(B) \leq \langle \gamma, \Gamma_{A, \ell}, \Trans(A) \rangle
\]
Since $\Gamma_{B, i} = \Gamma_{A, i}$ for all $i\neq \ell$, it follows that
\[
\Trans(B) = \langle \Trans( \langle \gamma, \Gamma_{A, \ell} \rangle)\;, \; \Trans(A) \rangle \leq
\langle \gamma, \Gamma_{A, \ell}, \Trans(A) \rangle
\]
\end{proof}

\subsubsection{Proof of \propref{Mgammarank}}\seclab{Mgammarank-proof}
We check the rank function axioms \cite[Section 1.3]{O11}.

\noindent\textbf{Non-negativity:} This follows from the fact that $\rep_{\Gamma_k}(\cdot)$ is non-negative, and the
sum of the $\frac{1}{2}T(\cdot)$ terms cannot exceed $n$.

\noindent\textbf{Monotonicity:} This is immediate from \lemref{rad-monotone} and the fact that $-T(\cdot)$ and $\rep_{\Gamma_k}(\Lambda(\cdot))$
only increase when the size of the closure increases.

\noindent\textbf{Normalization:} To prove that $g_1$ is normalized, let $A\subset E_{\Gamma_k,n}$ and $B = A + (\gamma,\ell)$.  Since all the
$T(\Gamma'_{\cdot,i})$ terms cancel except for the ones with $i=\ell$, the increase is given by
\[
g_1(B) - g_1(A) =
\frac{1}{2} \left(
\rep_{\Gamma_k}(\Trans(B)) - \rep_{\Gamma_k}(\Trans(A)) - T(\Gamma_{B, \ell}) + T(\Gamma_{A, \ell})
\right)
\]
Because the r.h.s. is an invariant of the closure by \propref{rad-rep-minus-T}, we pass to closures and apply
\lemref{reduce1} to see that the r.h.s. is equal to
\[
\frac{1}{2}
\left(
\rep_{\Gamma_k}(\Trans(\Gamma_{B, \ell}')) - \rep_{\Gamma_k}(\Trans(\Gamma_{A, \ell}'))  -
T(\Gamma_{B, \ell}') + T(\Gamma_{A, \ell}')
\right)
\]
Using \lemref{reduce2} then tells us that this can be simplified further to
\[
\frac{1}{2}
\left(
\rep_{\Gamma_k}(\Trans(\langle \gamma, \Gamma_{A, \ell}'\rangle)) - \rep_{\Gamma_k}(\Trans(\Gamma_{A, \ell}'))  -
T(\langle \gamma, \Gamma_{A, \ell}'\rangle) + T(\Gamma_{A, \ell}')
\right)
\]
at which point \propref{rad-rep-minus-T} applies, and we conclude that the increase is either zero or one.

\noindent\textbf{Submodularity:}
We will verify the following form of the submodular inequality:
\begin{eqnarray}\label{local-submodular}
f(A\cup \{(\gamma, \ell)\}) - f(A) \ge f(B\cup \{(\gamma, \ell)\}) - f(B) & \text{\qquad for all $A\subset B$}
\end{eqnarray}
Inspecting the argument for normalization and \propref{rad-rep-minus-T}, we
see that the r.h.s., is positive only if $\gamma \notin \Rad(\Gamma_{B,\ell}))$, in which case it
is always $2$.  By \lemref{rad-monotone}, for this $\gamma$, we also have
$\gamma \notin \Rad(\Gamma_{A,\ell})$, so the l.h.s. is also $2$.  Because both sides are
always non-negative, \eqref{local-submodular} follows.
\eop

\subsubsection{The bases and independent sets}\seclab{group-matroid-sets}
With the rank function of $M_{\Gamma_k,n}$ determined, we can give a structural characterization of its bases
and independent sets.  Let $A\subset E_{\Gamma_k,n}$.  We define $A$ to be \emph{independent} if
\[
|A| = g_1(A)
\]
If $A$ is independent and, in addition
\[
|A| = c(A) + \frac{1}{2} \rep_{\Gamma_k}(\Trans(\Gamma_k))
\]
we define $A$ to be \emph{tight}.  A (not-necessarily independent) set $A$ with $c(A)$ parts that contains a tight
subset on $c(A)$ is defined to be \emph{spanning}.

We define the classes
\begin{eqnarray*}
\mathcal{B}(M_{\Gamma_k,n}) & =   &
\left\{
B\subset E_{\Gamma_k,n} : \text{$B$ is independent and $|B| = n + \frac{1}{2} \rep_{\Gamma_k}(\Gamma_k)$}
\right\} \\
\mathcal{I}(M_{\Gamma_k,n}) & = &
\left\{
B\subset E_{\Gamma_k,n} : \text{$B$ is independent}
\right\}
\end{eqnarray*}
It is now immediate from \propref{Mgammarank} that:
\begin{lemma}\lemlab{Mgammasets}
The classes $\mathcal{I}(M_{\Gamma_k,n})$ and $\mathcal{B}(M_{\Gamma_k,n})$ are the independent sets and bases of
the matroid $M_{\Gamma_k,n}$.
\end{lemma}

\subsubsection{Structure of tight sets}
We also have a structural characterization of the tight independent sets in $M_{\Gamma_k,n}$.
\begin{lemma}\lemlab{Mgammatight}
An independent set $A\in \mathcal{I}(M_{\Gamma_k,n})$ is tight if and only if it is one of two types:
\begin{itemize}
\item [\textbf{(A)}]  Each of the non-empty $A_i$ contains a rotation.
One exceptional non-empty $A_i$
contains $\frac{1}{2}\rep_{\Gamma_k}(\Trans(\Gamma_k))$ additional elements,
and $\rep_{\Gamma_k}(\Trans(\Gamma_{A,i}))=\rep_{\Gamma_k}(\Trans(\Gamma_k))$,
and all the rest of the $A_i$ contain a single rotation only.
\item [\textbf{(B)}] Each of the non-empty $A_i$ contains a rotation.  Two exceptional non-empty $A_i$
(w.l.o.g., $A_1$ and $A_2$)	contain, between them, $\frac{1}{2}\rep_{\Gamma_k}(\Trans(\Gamma_k))$ additional
elements and
$$\rep_{\Gamma_k}(\langle\Trans(\Gamma_{A,1}), \Trans(\Gamma_{A,2})\rangle)=\rep_{\Gamma_k}(\Trans(\Gamma_k)).$$
\end{itemize}
Type \textbf{(B)} is only possible when $\Gamma_k=\Gamma_2$.
\end{lemma}
\begin{proof}
One direction is straightforward: A set $A\subset E_{\Gamma_k,n}$ of either type
\textbf{(A)} or \textbf{(B)} satisfies, by hypothesis, $|A|=c(A)+\frac{1}{2}\rep_{\Gamma_k}(\Trans(\Gamma_k))$;
by construction $T(\Gamma_{A,i})$ is zero for all the non-empty
$A_i$ and $\rep_{\Gamma_k}(\Trans(A)) = \rep_{\Gamma_k}(\Trans(\Gamma_k))$.

On the other hand, assuming that $A$ is tight, we see that each non-empty part has to contain a rotation,
and, since $A$ is independent there are only one (for $k=3,4,6$) or
two ($k=2$) additional elements in $A$. For $k=3,4,6$, the single $A_i$ containing the extra element must generate
a translation in which case $\rep_{\Gamma_k}(\Trans(\Gamma_{A,i}))=\rep_{\Gamma_k}(\Trans(\Gamma_k))$.
For $k=2$, the translation subgroups of the $A_i$ containing the extra elements must generate a rank $2$ translation
subgroup of $\Lambda(\Gamma_k)$, and the desired conclusion follows.
\end{proof}

\subsubsection{Conjugation of independent sets}\seclab{Mkn-conj}
Let $A\in \mathcal{I}(M_{\Gamma_k,n})$ be an independent set, and suppose, w.l.o.g.,
that $A_1, A_2,\ldots, A_{c(A)}$ are the non-empty parts of $A$.  Let $\gamma_1,\gamma_2,\ldots,
\gamma_{c(A)}$ be elements of $\Gamma_k$.  The
\emph{conjugation of $A$ by $\gamma_1,\gamma_2,\ldots, \gamma_{c(A)}$} is defined to be
\[
\left\{(\gamma_i^{-1}A_i\gamma_i, i) : 1\le i\le c(A) \right\}
\]
\begin{lemma}\lemlab{Mgammaconj}
Let $A\in \mathcal{I}(M_{\Gamma_k})$ be an independent set.  Then the conjugation of $A$ by $c(A)$
elements $\gamma_1,\ldots, \gamma_{c(A)}$ is also independent.
\end{lemma}
\begin{proof}
\lemref{rad-conj} implies that the closure of translation subgroups is preserved under conjugation, and
whether or not $A_i$ contains a rotation is preserved as well. Since the rank function $g_1$ is determined by these two
properties of the $A_i$, we are done.
\end{proof}

\subsubsection{Separating and fusing independent sets}\seclab{Mkn-fuse}
Let $A\in \mathcal{I}(M_{\Gamma_k})$ be an independent set.  A \emph{separation of $A$} is defined to be the following operation:
\begin{itemize}
\item Select $i$ and $j$ such that $A_j$ is empty.
\item Select a (potentially empty) subset $A'_i\subset A_i$ of $A_i$.
\item Replace all elements $(\gamma,i)\in A'_i$ with $(\gamma,j)$.
\end{itemize}
\begin{lemma}\lemlab{Mgammasep}
Let $A\in \mathcal{I}(M_{\Gamma_k})$ be an independent set.  Then any separation of $A$ is also an independent set.
\end{lemma}
\begin{proof}
Let $B$ be a separation of $A$.  If the subset $A'_i$ in the definition of a separation is
empty, then $B$ is the same as $A$, and there is nothing to prove.

An independent set is either tight or a subset of a tight set.  (Bases in particular are tight.)
Consequently, by \lemref{Mgammatight}, either $B_i$ or $B_j$ consists of a single element.
Assume w.l.o.g., it is $B_j$.  Define $C \subset E_{\Gamma_k, n}$ as
$C_k = B_k$ for $k \neq j$ and $C_j$ empty; i.e. $C$ is $B$ with the single element in $B_j$ dropped.
Then $C$ is a subset of $A$ and hence independent.  If $B_j$ consists of a rotation, then adding it to $C$
clearly preserves independence.   If $B_j$ consists of a translation $\gamma$, then, since $A$ is independent,
we must have $\gamma \notin \Rad(\Trans(C))$.  Consequently $B = C + (\gamma, j)$ is independent
since $\Rad(\Trans(B)) \gneqq \Rad(\Trans(C))$ and hence $\rep_{\Gamma_k}(\Trans(B)) > \rep_{\Gamma_k}(\Trans(C))$.
\end{proof}

The reverse of separation is \emph{fusing a set $A$ on $A_i$ and $A_j$}.  This operation replaces $A_i$ with $A_i\cup A_j$ and
makes $A_j$ empty.  Fusing does not, in general, preserve independence, but it takes tight sets to spanning ones.
\begin{lemma}\lemlab{Mgammafuse}
Let $A$ be a tight independent set, and suppose that $A_i$ and $A_j$ are non-empty.  Then, after fusing $A$ on $A_i$ and $A_j$,
the result is a spanning set (with one less part).
\end{lemma}
\begin{proof}
Let $B$ be the set resulting from fusing $A$ on $A_i$ and $A_j$. By hypothesis, all the non-empty $A_\ell$ contain a rotation,
so this is true of the non-empty $B_\ell$ as well.
The lemma then follows by noting that $\Trans(A) \le \Trans(B)$, so the same is true of the closures by \lemref{rad-monotone}.
Thus, $g_1(B)=c(B)+\rep_{\Gamma_k}(\Trans(\Gamma_k))$, and this implies $B$ is spanning.
\end{proof}

\section{Matroidal sparse graphs}\seclab{bigsec-graphs}
\subsection{Colored graphs and the map $\rho$}\seclab{colored-graphs}
We will use \emph{colored graphs}, which are also known as ``gain graphs'' (e.g., \cite{R11}) or ``voltage graphs''
(e.g. \cite{Z98}) as the combinatorial model for crystallographic frameworks and direction networks.  In this section we give the
definitions and explain the relationship between colored graphs and graphs with a free $\Gamma_k$-action.

\subsubsection{Colored graphs}
Let $G=(V,E)$ be a finite, directed graph, with $n$ vertices
and $m$ edges.  We allow multiple edges and self-loops, which are treated the same as other edges.
A \emph{$\Gamma_k$-colored-graph} (shortly, \emph{colored graph}) $(G,\bgamma)$ is a finite, directed multigraph $G$ and an assignment
$\bgamma=(\gamma_{ij})_{ij\in E(G)}$ of a group element $\gamma_{ij}\in \Gamma_k$ (the ``color'') to each edge $ij\in E(G)$.

\subsubsection{The covering map} Although we work with colored graphs because they are technically easier, crystallographic
frameworks were defined in terms of infinite graphs $\Gtilde$ with $\Gamma_k$ acting freely and with
finite quotient by the representation
$\varphi : \Gamma_k\to \operatorname{Aut}(\tilde{G})$.  In fact, the formalisms are equivalent, via a specialization of
covering space theory (e.g.,  \cite[Section 1.3]{H02}).  We provide the dictionary here for completeness.

Let $(G,\bgamma)$ be a colored graph, we define its \emph{lift} $\Gtilde=(\Vtilde,\Etilde)$ by the following construction:
\begin{itemize}
\item For each vertex $i\in V(G)$, there is a subset of vertices
$\{i_{\gamma}\}_{\gamma\in \Gamma}\subset V(\Gtilde)$ (the fiber over $i$).
\item For each (directed) edge $ij\in E(G)$ with color $\gamma_{ij}$, and for each $\gamma\in \Gamma_k$,
there is an edge $i_{\gamma}j_{\gamma\cdot\gamma_{ij}}$
in $E(\Gtilde)$ (the fiber over $ij$).
\item The $\Gamma$-action on vertices is $\gamma \cdot i_{\gamma'} = i_{\gamma \gamma'}$.
The action on edges is that induced by the vertex action.
\end{itemize}

Now let $(\Gtilde,\varphi)$ be an infinite graph with a free $\Gamma_k$-action that has finite quotient.
We associate a colored graph $(G,\bgamma)$ to $(\Gtilde,\varphi)$ by the following construction, which we define to be
a \emph{colored quotient}:
\begin{itemize}
\item Let $G=\Gtilde/\Gamma$ be the quotient of $\Gtilde$ by $\Gamma$, and fix an (arbitrary) orientation
of the edges of $G$ to make it a directed graph.  By hypothesis, the vertices of $G$ correspond to the vertex orbits in $\tilde{G}$
and the edges to the edge orbits in $\tilde{G}$
\item For each vertex orbit under $\Gamma$ in $\Gtilde$, select a representative $\tilde{i}$.
\item For each edge orbit in $\Gtilde$ there is a unique edge $\widetilde{ij}$ that has the representative
$\tilde{i}$ as its tail.  There is also a unique element $\gamma_{ij}\in \Gamma$ such that the
head of $\widetilde{ij}$ is $\gamma_{ij}(\tilde{j})$.  We define this $\gamma_{ij}$ to be the color
on the edge $ij\in G$.
\end{itemize}
From the definition, we see that the specific colored quotient depends on the choice of representatives,
however they are all related as follows.  For any choice of representatives, the lift $\tilde{G}'$
is isomorphic to $\tilde{G}$ as a graph, and this isomorphism is $\varphi$-equivariant.  It the
follows that the lifts of any two colored quotients are isomorphic to each other via a
$\varphi$-equivariant map.

The \emph{projection map} from $(\Gtilde,\varphi)$ to its colored quotient is the function that sends
a vertex $\tilde{i}\in V(\Gtilde)$ its representative $i\in V(G)$.  Figures
\ref{fig:gam2graphtoperiodic} and \ref{fig:gam4graphtoperiodic} both show examples; the
color coding of the vertices in the infinite developments indicate the fibers over
vertices in the colored quotient.  The discussion above shows:
\begin{lemma}\lemlab{lifts}
Let $(G,\bgamma)$ be a $\Gamma_k$-colored graph.  Then its lift is well defined, and is an infinite
graph with a free $\Gamma_k$-action.  If $(\Gtilde,\varphi)$ is an infinite graph with a free
$\Gamma_k$-action, then it is the lift of any of its colored quotient, and the projection map
is well-defined and a covering map.
\end{lemma}

\subsubsection{The map $\rho$} \seclab{rho}
Let $(G,\bgamma)$ be a colored graph, and let $P=\{e_1, e_2,\ldots, e_t\}$ be any \emph{closed path}
in $G$; i.e., $P$ is a (not necessarily simple) walk in $G$ that starts and ends at the same vertex crossing
the edges $e_i$ in order.  If we select a vertex $b$ as a \emph{base point}, then the closed paths represent elements of the
\emph{fundamental group} $\pi_1(G,b)$.

We define the map
$\rho$ as:
\[
\rho(P) = \gamma_{e_1}^{\epsilon_1}\cdots \gamma_{e_t}^{\epsilon_t}
\]
where $\epsilon_i$ is $1$ if $P$ crosses $e_i$ in the forward direction (from tail to head) and $-1$ otherwise.  For a connected graph
$G$ and choice of base vertex $i$, the map $\rho$ induces a well-defined homomorphism $\rho: \pi_1(G, i) \to \Gamma$.

\subsection{The subgroup of a $\Gamma_k$-colored graph}\seclab{closed-paths}
The map $\rho$, defined in the previous section, is fundamental to the results of this paper.  In
this section, we develop properties of the $\rho$-image of a colored graph $(G,\bgamma)$ and
connect it with the matroid $M_{\Gamma_k,n}$ which was defined in \secref{group-matroid}.

\subsubsection{Colored graphs with base vertices}
Let $(G,\bgamma)$ be a colored graph with $n$ vertices and $c$ connected components $G_1,G_2,\ldots, G_c$.
We select a \emph{base vertex} $b_i$ in each connected component $G_i$, and denote the set of
base vertices by $B$.  The triple $(G,\bgamma,B)$ is then defined to be a \emph{marked colored graph}.

If $(G,\bgamma,B)$ is a marked colored graph then $\rho$ induces a homomorphism from $\pi_1(G_i,b_i)$
to $\Gamma_k$.  In the rest of this section, we show how to use these homomorphisms to
define a map from $(G,\bgamma)$ to $E_{\Gamma_k,n}$, the ground set of the matroid $M_{\Gamma_k,n}$.

\subsubsection{Fundamental closed paths generate the $\rho$-image}
Let $(G,\bgamma,B)$ be a marked colored graph with $n$ vertices and $c$ connected components.  Select
and fix a maximal forest $F$ of $G$, with connected components $T_1,T_2,\ldots, T_c$.  The $T_i$
are spanning trees of the connected components $G_i$ of $G$, with the convention that when a connected
component $G_i$ has no edges there is a one-vertex ``empty tree'' $T_i$.

With this data, we define, for each edge $ij\in E(G)-E(F)$ the \emph{fundamental closed path of $ij$} to be the path that:
\begin{itemize}
\item Starts at the base vertex $b_\ell$ in the same connected component $G_\ell$ as $i$ and $j$.
\item Travels the unique path in $T_\ell$ to $i$.
\item Crosses $ij$.
\item Travels the unique path in $T_\ell$ back to $b_\ell$.
\end{itemize}
Fundamental closed paths with respect to $F$ in $G_i$ generate $\pi_1(G_i, b_i)$ by \cite[Proposition 1A.2]{H02}.

\subsubsection{From colored graphs to sets in $E_{\Gamma_k,n}$}
We now let $(G,\bgamma,B)$ be a marked colored graph and fix a choice of spanning forest $F$.  We associate
with $(G,\bgamma,B,F)$ a subset $A(G,B,F)$ of $E_{\Gamma_k,n}$ (defined in \secref{group-matroid}) as follows:
\begin{itemize}
\item For each edge $ij\in E(G_\ell)-E(T_\ell)$, let $P_{ij}$ be the fundamental closed  path of $ij$
with respect to $T_\ell$ and $b_\ell$.
\item Add an element $(\rho(P_{ij}),\ell)$ to $A(G,B,F)$.
\end{itemize}
The following is immediate from the previous discussion.
\begin{lemma}\lemlab{GBF-eq-GammaA}
Using the notation from \secref{group-matroid}, $\Gamma_{A, \ell} = \rho(\pi_1(G_\ell, b_\ell))$
where $A = A(G, B, F)$.
\end{lemma}
Since we will show, in \secref{gamma22}, that the invariants we need are independent of $B$ and $F$,
we frequently suppress them from the notation when the context is clear.

\subsection{$\Gamma$-$(2,2)$ graphs} \seclab{gamma22}
In this section we define \emph{$\Gamma$-$(2,2)$ graphs} which are the first of two key
families of colored graphs introduced in this paper (the second is \emph{$\Gamma$-colored-Laman
graphs}, defined in \secref{gamma-laman}).  We also state the main combinatorial results on
$\Gamma$-$(2,2)$ graphs, but defer the proof of a key technical result, \propref{gamma11},
to \secref{sparse}.

\subsubsection{The translation subgroup of a colored graph}\seclab{translation-subgroup-colored-graph}
Let $(G,\bgamma,B)$ be a marked colored graph, as in \secref{closed-paths}, with connected components
$G_1,G_2,\ldots, G_c$ and base vertices $b_1,b_2,\ldots, b_c$.  Recall from \secref{rho}
that, with this data, there are homomorphisms
\[
\rho : \pi_{1}(G_i,b_i)\to \Gamma_k
\]
We define $\Trans(G,B)$ to be
\[
\Trans(G,B) = \langle \Trans(\rho(G_i,b_i)) : i = 1,2,\ldots, c \rangle
\]
We define $\rep_{\Gamma_k}(G)=\rep_{\Gamma_k}(\Trans(G,B))$.  As the notation suggests,
$\rep_{\Gamma_k}(G)$ is independent of the choice of base vertices $B$.
\begin{lemma}\lemlab{rep-of-trans-G}
Let $(G,\bgamma,B)$ be a marked colored connected graph.
The quantity  $\rep_{\Gamma_k}(G)$ is independent of the choice of base vertices, and so is
a property of the underlying colored graph $(G,\bgamma)$.
\end{lemma}
\begin{proof}
Changing base vertices corresponds to conjugation.  \lemref{rad-conj} implies that
the closure of $\Trans(G,B)$ is preserved under conjugation.  Since $\rep_{\Gamma_k}(\cdot)$
depends only on the closure, the lemma follows.
\end{proof}

\subsubsection{The quantity $T$ for a colored graph}
Let $(G,\bgamma,B)$ be a marked colored graph, with $G$ connected (and so a single
base vertex $b$).  We define $T(G)$ to be $T(\rho(\pi_1(G,b)))$.  The proof of the
following lemma is similar to that of \lemref{rep-of-trans-G}.
\begin{lemma}\lemlab{T-of-G}
Let $(G,\bgamma,B)$ be a marked colored graph.
The quantity  $T(G)$ is independent of the choice of base vertices, and so is
a property of the underlying colored graph $(G,\bgamma)$.
\end{lemma}
\begin{figure}[htbp]
\centering
\includegraphics[width=.6\textwidth]{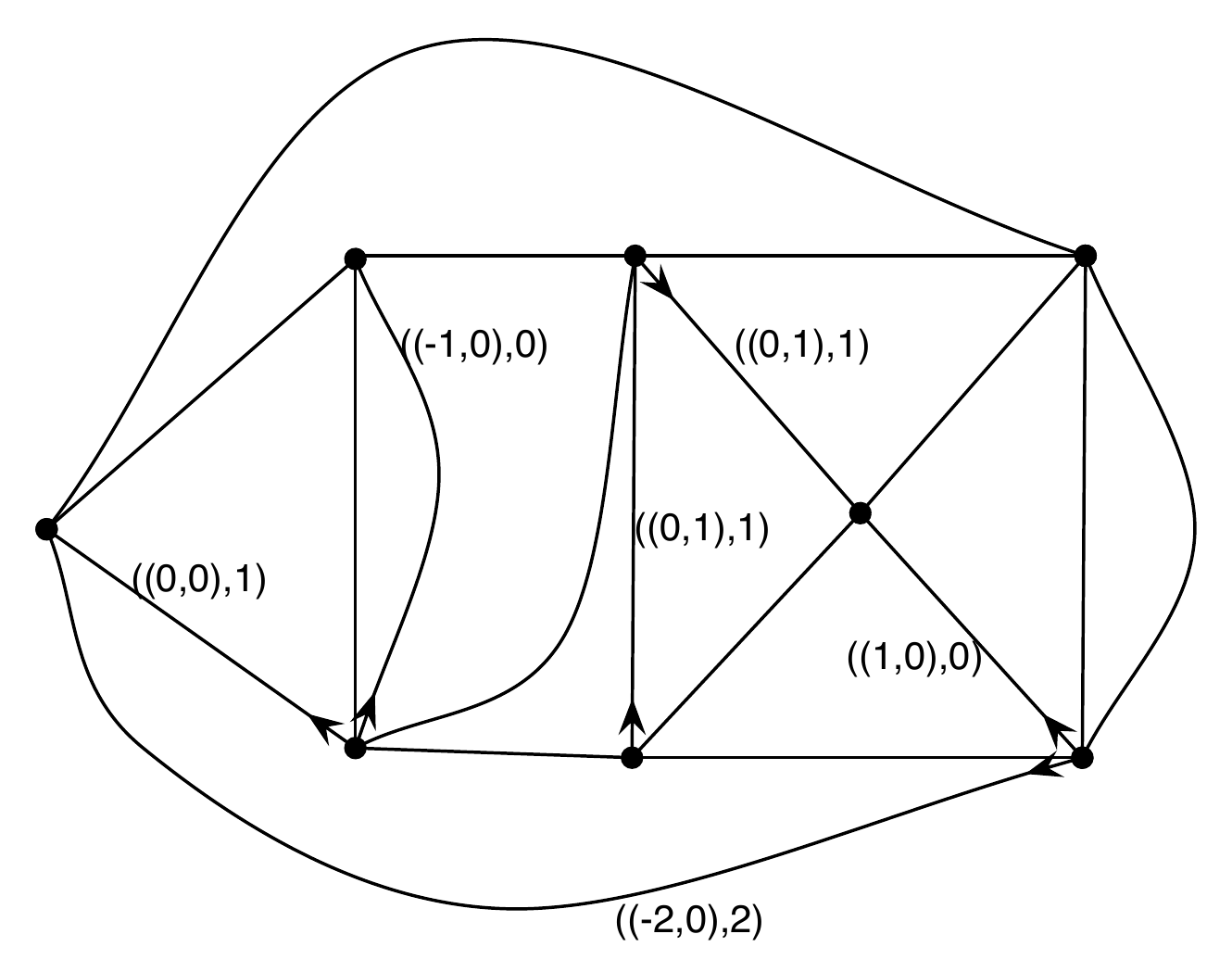}
\caption{An example of a \GammaTT graph when $\Gamma = \Gamma_3$.}
\label{fig:gam22}
\end{figure}
\subsubsection{$\Gamma$-$(2,2)$ graphs}
We are now ready to define $\Gamma$-$(2,2)$ graphs.  Let $(G,\bgamma)$ be a colored graph
with $n$ vertices and $c$ connected components $G_i$.  We define the function $f$ to be
\[
f(G) = 2n + \rep_{\Gamma_k}(G) - \sum_{i=1}^c T(G_i)
\]
A colored graph $(G,\bgamma)$ on $n$ vertices and $m$ edges is defined to be a \emph{$\Gamma$-$(2,2)$ graph} if:
\begin{itemize}
\item The number of edges $m$ is $2n+\rep_{\Gamma_k}(\Trans(\Gamma_k))$ (i.e., it is the maximum possible value for $f$).
\item For every subgraph $G'$ of $G$, with $m'$ edges, $m' \le f(G')$.
\end{itemize}
We note that it is essential that the definition is made
over \emph{all} subgraphs, and not just vertex-induced or connected ones.
\figref{gam22} shows an example of a \GammaTT-graph.

\subsubsection{Direction network derivation}
Before continuing with the development of the combinatorial theory, we quickly motivate
the definition of \GammaTT-graphs.  Readers who are not familiar with rigidity and
direction networks may want to either skip to \secref{gamma11-def} and revisit this,
purely informative, section after reading the definitions in
\secref{bigsec-dn}.

\propref{crystal-collapse}, in \secref{cdns} below,
implies that a generic direction network on a $\Gamma_k$-colored graph $(G,\bgamma)$ has
only \emph{collapsed} realizations (with all the points on top of each other and a trivial representation for
$\Trans(\Gamma_k)$) if and only if $(G,\bgamma)$ has a spanning \GammaTT subgraph.

The definition of the function $f$ comes from analyzing the degrees of freedom in realizations which have the endpoints
of each edge coincident (these are collapsed when $G$ is connected).
For any realization $G(\vec p, \Phi)$, we can translate it (this preserves directions),
so that $\Phi(r_k)$ has the origin as its rotation center.  Then, restricted to a subgraph $G'$ of $G$:
\begin{itemize}
\item The total number of variables involved in the equations giving the edge directions is $2n'+\rep_{\Gamma_k}(G')$.
Since we fix $\Phi(r_k)$ to rotate at the origin (see \secref{cdns} for an
explanation why we can do this), the only variability
left in $\Phi$ is $\Phi(\Trans(\Gamma_k))$.  Since $\rep_{\Gamma_k}(G')$
measures how much of $\Trans(\Gamma)$ is ``seen'' by $G'$, this is the term we add.
\item Each connected component $G_i$ has a $T(G'_i)$-dimensional space of collapsed realizations.  If $G_i'$ has a
rotation, then a collapsed realization of the lift $\tilde G_i'$ must lie on the corresponding rotation center
since a solution must be rotationally symmetric.  When $G_i'$ has no rotation, no such restriction exists, and
there are $2$-dimensions worth of places to put the collapsed $\tilde G_i'$.
Each collapsed connected component is independent of the others, so this term is additive
over connected components.
\end{itemize}
The heuristic above coincides with the definition of the function $f$.

\subsubsection{Map-graphs} \seclab{sparse-prelim}
In this section we recall the definition of a map-graph. As we will see in the next section, the structure of
\GammaTT graphs is closely related to map-graphs. A \emph{map-graph} is a graph in which every connected
component has exactly one cycle.  In this definition, self-loops
correspond to cycles.
A \emph{$2$-map-graph} is a graph that is the edge-disjoint union of two spanning
map-graphs.  Observe that map-graphs, and, consequently, $2$-map-graphs, do \emph{not} need
to be connected.

\subsubsection{$\Gamma$-$(1,1)$ graphs}\seclab{gamma11-def}
We will characterize $\Gamma$-$(2,2)$ graphs in terms of decompositions into simpler
\emph{$\Gamma$-$(1,1)$ graphs}\footnote{The terminology of ``$(2,2)$'' and ``$(1,1)$'' comes from
the fact that spanning trees of finite graphs are ``$(1,1)$-tight'' in the sense of \cite{LS08}.  The
$\Gamma$-$(1,1)$ graphs defined here are, in a sense made more precise in \cite[Section 5.2]{MT13},
analogous to spanning trees.}, which we now define.

Let $(G,\bgamma)$ be a colored graph and select a base vertex $b_i$ for each connected component $G_i$ of $G$.
We define $(G,\bgamma)$ to be a \emph{$\Gamma$-$(1,1)$ graph} if:
\begin{itemize}
\item $G$ is a map-graph plus $\frac{1}{2}\rep_{\Gamma_k}(\Trans(\Gamma_k))$ additional edges.
\item For each connected component $G_i$ of $G$, $\rho(\pi_1(G_i,b_i))$ contains a rotation.
\item We have $\rep_{\Gamma_k}(G)=\rep_{\Gamma_k}(\Trans(\Gamma_k))$, or equivalently, $\Rad(\Trans(G,B)) = \Trans(\Gamma_k)$.
\end{itemize}
Although we do not define \GammaOO graphs via sparsity counts, there is an alternative
characterization in these terms.  We define the function $g(G)$ to be
\[
g(G) = n + \frac{1}{2}\rep_{\Gamma_k}(G) - \frac{1}{2}\sum_{i=1}^c T(G_i)
\]
where $(G,\bgamma)$ is a colored graph and $n$ and $c$ are the number of vertices and
connected components.  Notice that $g=\frac{1}{2}f$.  In \secref{sparse} we will show:
\begin{restatable}{prop}{gammaOOmatroid} \proplab{gamma11}
The family of $\Gamma$-$(1,1)$ graphs gives the bases of a matroid, and the rank of
the $\Gamma$-$(1,1)$ matroid is given by the function:
\[
g(G) = n + \frac{1}{2}\rep_{\Gamma_k}(G) - \frac{1}{2}\sum_{i=1}^c T(G_i)
\]
In particular, this implies that $g$ is non-negative, submodular, and monotone.
\end{restatable}
\subsubsection{Decomposition characterization of $\Gamma$-$(2,2)$ graphs}\seclab{cone22-decomp}
The key combinatorial result about $\Gamma$-$(2,2)$ graphs, that is used in an essential way to
prove the ``collapsing lemma'' \propref{crystal-collapse}, is the following.
\begin{prop}\proplab{gamma22-decomp}
Let $(G,\bgamma)$ be a colored graph.  Then $(G,\bgamma)$ is a $\Gamma$-$(2,2)$ graph if and only
if it is the edge-disjoint union of two spanning $\Gamma$-$(1,1)$ graphs.
\end{prop}
\begin{proof}
Since $f=2g$, \propref{gamma11} implies that $g$ meets the hypothesis required for the
Edmonds-Rota construction \cite{ER66}, from which we conclude that $\Gamma$-$(2,2)$ graphs
are a matroidal family.  The existence of the desired decomposition follows from the Matroid
Union Theorem for rank functions \cite{ER66}.
\end{proof}

\subsection{$\Gamma$-colored Laman graphs}\seclab{gamma-laman}
We are now ready to define \emph{$\Gamma$-colored-Laman graphs}, which are the
colored graphs characterizing minimally rigid generic frameworks in \theoref{main}.
Just as for \GammaTT graphs, we define them via sparsity counts.
\begin{figure}[htbp]
\centering
\subfigure[]{\includegraphics[width=.45\textwidth]{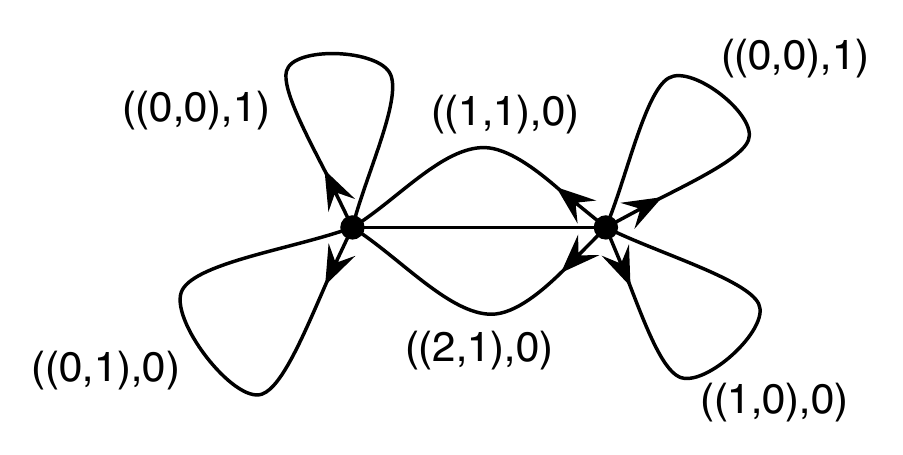}}
\subfigure[]{\includegraphics[width=.35\textwidth]{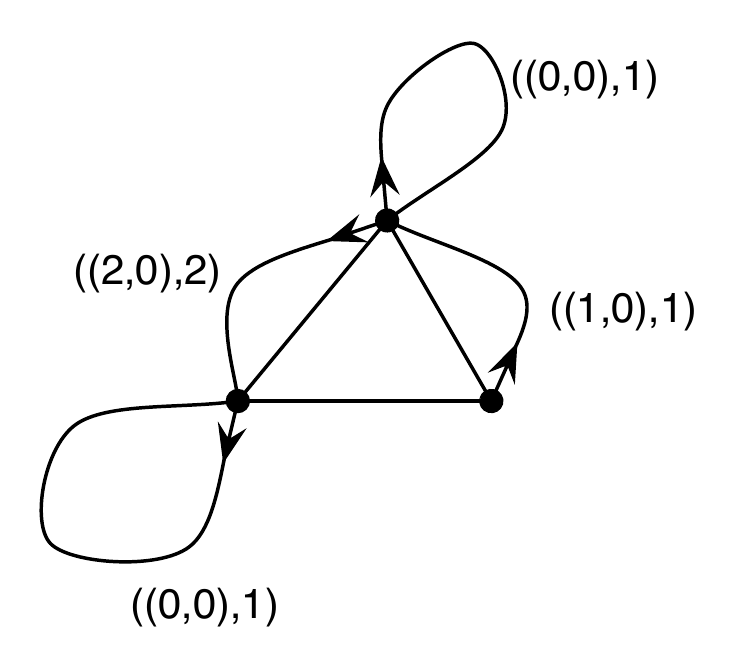}}
\caption{Examples of \GammaCL graphs:
(a) a $\Gamma_2$-colored-Laman graph;
(b) a $\Gamma_3$-colored-Laman graph
}
\label{fig:gamlaman}
\end{figure}
\subsubsection{Definition of $\Gamma$-colored-Laman graphs}
Let $(G,\bgamma)$ be a colored graph, and let $f$ be the sparsity function defined in \secref{gamma22}.
The most direct definition of the sparsity function $h$ for $\Gamma$-colored-Laman graphs is:
\[
h(G) = f(G) -1
\]

A colored graph $(G,\bgamma)$ is defined to be \emph{$\Gamma$-colored-Laman} if:
\begin{itemize}
\item $G$ has $n$ vertices and $m=2n + \rep_{\Gamma_k}(\Trans(\Gamma_k)) - T(\Gamma_k) - 1$ edges.
\item For all subgraphs $G'$ spanning $m'$ edges, $m'\le h(G')$
\end{itemize}
\figref{gamlaman} shows some examples of $\Gamma$-colored-Laman graphs.
If a colored graph is a subgraph of a $\Gamma$-colored-Laman graph, then it is defined to be
\emph{$\Gamma$-colored-Laman sparse}.  Equivalently, $(G,\bgamma)$ is $\Gamma$-colored-Laman sparse
when the condition ``$m'\le h(G')$'' above holds for all subgraphs $G'$.

\subsubsection{Alternate formulation of $\Gamma$-colored-Laman graphs}
While the definition of $h$ is all that is needed to prove
\theoref{main}, it does not give any motivation in terms of a degree-of-freedom count.  We now
give an alternate formulation of $\Gamma$-colored-Laman via the Teichm\"uller space and the centralizer, which
were defined in \secref{teich}, that will let us do this.

Let $(G,\bgamma,B)$ be a marked colored graph with connected components $G_1, \dots, G_c$ and $n$ vertices,
and let $\Trans(G,B)$ be its translation subgroup as defined in \secref{translation-subgroup-colored-graph}.
We define
\[
\teich_{\Gamma_k}(G) = \teich_{\Gamma_k}(\Trans(G, B))
\]
which, by a proof nearly identical to that of \lemref{rep-of-trans-G}, is well-defined and independent of the
choice of base vertices.

For a component $G_\ell$ with base vertex $b_\ell$, we set
$\cent_{\Gamma_k}(G_\ell) = \cent_{\Gamma_k}(\rho(\pi_1(G_\ell, b_\ell)))$.  For similar reasons,
$\cent_{\Gamma_k}(G_\ell)$ is also independent of the base vertex.

We can now define a ``more natural'' sparsity function
\[
h'(G) = 2n + \teich_{\Gamma_k}(G) -
\left(
\sum_{i=1}^c \cent_{\Gamma_k}(G_i))
\right)
\]
The class of colored graphs defined by $h'$ is the same as that arising from $h$, giving a second definition
of $\Gamma$-colored-Laman graphs.  Since \lemref{gamma-cl-equiv} is not used to prove any further results, we
omit the proof.
\begin{lemma}\lemlab{gamma-cl-equiv}
A colored graph $(G,\bgamma)$ is \GammaCL if and only if:
\begin{itemize}
\item $G$ has $n$ vertices and $m=2n+\teich(\Gamma)-\cent(\Gamma)$ edges.
\item For all subgraphs $G'$ spanning $m'$ edges, $m'\le h'(G')$
\end{itemize}
\end{lemma}

\subsubsection{Degree of freedom heuristic}
The function $h'$ is amenable to an interpretation that allows us, by \lemref{gamma-cl-equiv}, to
give a  rigidity-theoretic ``degree of freedom'' derivation of $\Gamma$-colored-Laman graphs.  This
section is expository, and readers unfamiliar with rigidity theory may skip to \secref{doubling}
and return here after reading \secref{bigsec-rigidity}.

Given a framework with underlying colored graph $(G,\bgamma)$, with
the graph $G$ having $n$ vertices and $c$ connected components $G_1,G_2,\ldots, G_c$, we find that:
\begin{itemize}
\item We have $2n$ degrees of freedom from the points.  From the representation
$\Phi: \Gamma_k \to \Euc(2)$,
there are $\rep(\Gamma_k)$ degrees of freedom,
but if we mod out by trivial motions from $\Euc(2)$, we have $\teich_{\Gamma_k}(\Gamma)$ degrees of
freedom left.  	However, we have only $\teich_{\Gamma_k}(G)$ degrees of freedom that apply to $G$.
\item Each connected component has $\cent_{\Gamma_k}(G_i)$ trivial degrees of freedom.  Since elements in the
centralizer for $G_i$ commute with those in $\rho(\pi_1(G_i))$,
we may ``push the vertices of $G_i$ around'' with the centralizer
elements while preserving symmetry.  Since these motions always exist, they are trivial.
\end{itemize}
This heuristic corresponds to the function $h'$.

\subsubsection{Edge-doubling characterization of $\Gamma$-colored-Laman graphs}\seclab{doubling}
The main combinatorial fact about $\Gamma$-colored-Laman graphs we need is the following simple
characterization by edge-doubling (cf. \cite{LY82,R84}).

\begin{prop}\proplab{doubleedge}
Let $\Gamma=\Gamma_k$ for $k=2,3,4,$ or $6$ be a crystallographic group and let $(G,\bgamma)$ be a
$\Gamma$-colored graph.  Then $(G,\bgamma)$ is $\Gamma$-colored-Laman if and only if for any
edge $ij\in E(G)$, the colored graph $(G',\bgamma')$ obtained by adding a copy of $ij$ to $G$
with the same color results in  a $\Gamma$-$(2,2)$ graph.
\end{prop}
\begin{proof}
This is straightforward to check once one notices that $(G,\bgamma)$ is
$\Gamma$-colored-Laman if and only if no subgraph $G'$ with $m'$ edges has $m'=f(G')$.
\end{proof}

\subsubsection{$\Gamma$-colored-Laman circuits}
Let $(G,\bgamma)$ be a colored graph.  We define $(G,\bgamma)$ to be a \emph{$\Gamma$-colored-Laman circuit}
if it is edge-wise minimal with the property of not being $\Gamma$-colored-Laman sparse.  More formally,
$(G,\bgamma)$ is a $\Gamma$-colored-Laman circuit if:
\begin{itemize}
\item $(G,\bgamma)$ is not $\Gamma$-colored-Laman sparse
\item For all colored edges $ij\in E(G)$, $(G-ij,\bgamma)$ is $\Gamma$-colored-Laman sparse
\end{itemize}
As the terminology suggests, $\Gamma$-colored-Laman circuits are the circuits of the matroid that has, as its bases,
$\Gamma$-colored-Laman graphs.  The following lemmas are immediate from the definition.
\begin{lemma}\lemlab{not-gamma-laman-sparse-implies-circuit}
Let $(G,\bgamma)$ be a colored graph.  If $(G,\bgamma)$ is not $\Gamma$-colored-Laman sparse, then
it contains a $\Gamma$-colored-Laman circuit as a subgraph.
\end{lemma}
\begin{lemma}\lemlab{gamma-laman-circuit-sparsity}
Let $(G,\bgamma)$ be a colored graph with $n$ vertices and $m$ edges.
Then $(G,\bgamma)$ is a $\Gamma$-colored-Laman circuit if and only if:
\begin{itemize}
\item The number of edges $m = f(G)$
\item For all subgraphs $G'$ of $G$, on $m'$ edges, $m' < f(G')$
\end{itemize}
\end{lemma}
Here, $f$ is the colored-$(2,2)$ sparsity function defined in \secref{gamma22}.

\subsection{\GammaOO{} graphs: proof of \propref{gamma11}} \seclab{sparse}
With the definitions and main properties of $\Gamma$-$(2,2)$ and $\Gamma$-colored-Laman
graphs developed, we prove:
\gammaOOmatroid*
\begin{proof}
The proposition follows immediately from Lemmas \ref{lemma:gamma11-bases} and \ref{lemma:gamma11-rank}, which are
proven below.
\end{proof}
With this, the proof of \propref{gamma22-decomp} is also complete.  The rest of this section is organized
as follows: first we prove that the $\Gamma$-$(1,1)$ graphs give the bases of a matroid and then
we argue that the rank function of this matroid is, in fact, the function $g$, defined in
\secref{gamma22}.

We recall from \secref{closed-paths}
that, for a marked colored graph $(G,\bgamma,B)$ with a fixed spanning forest
$F$, the map $\rho$, defined in \secref{colored-graphs}, induces a map from
$(G,\bgamma,B,F)$ to $E_{\Gamma_k,n}$, the ground set of the matroid $M_{\Gamma_k,n}$
from \secref{group-matroid}.  We adopt the notation of \secref{closed-paths}, and
denote the image of this map by $A(G,B,F)$.

We start by studying $A(G,B,F)$ in more detail.

\subsubsection{Rank of $A(G,B,F)$}
As defined, the set $A(G,B,F)$ depends on a choice of base vertices for each connected
component and a spanning forest $F$ of $G$.  Since we are interested in constructing a
matroid on colored graphs without additional data, the first structural lemma is that
the rank of $A(G,B,F)$ in $M_{\Gamma_k,n}$ is independent of the choices for $B$ and $F$.
\begin{lemma}\lemlab{rebase}
Let $(G,\bgamma,B)$ be a marked colored graph with connected components
$G_1,\ldots,$ $G_c$ and fix a spanning forest $F$ of $G$.  Then the rank of $A(G,B,F)$ in the matroid
$M_{\Gamma_k,n}$ is invariant under an arbitrary change of base vertices and spanning forest.
\end{lemma}
\begin{proof}
For convenience, shorten the notation $A(G,B,F)$ to $A$.
By \lemref{GBF-eq-GammaA} $\rho(\pi_1(G_\ell, v_\ell)) = \Gamma_{A, \ell}$.
Changing the spanning forest
$F$ just picks out a different set of generators for $\pi_1(G_\ell, v_\ell)$, and so does not
change $\Gamma_{A, \ell}$, and thus the rank in $M_{\Gamma_k,n}$, which does not
depend on the generating set, is unchanged.

To complete the proof, we show that changing the base vertices corresponds, in $E_{\Gamma_k,n}$,
to applying the conjugation operation defined in \secref{group-matroid} to $A$.
Suppose that $G$ is connected and fix a spanning tree $F$ and a base vertex $b$.
If $P$ is a closed path starting and ending at $b$, for any other vertex $b'$ there is a path $P'$ that:
starts at $b'$, goes to $b$ along a path $P_{bb'}$, follows $P$, and then returns from $b$ to $b'$
along $P_{bb'}$ in the other direction.  We have $\rho(P') = \rho(P_{bb'})\rho(P)\rho(P_{bb'})^{-1}$, so
$P$ and $P'$ have conjugate images.  Thus changing base vertices corresponds to conjugation, and by
\lemref{Mgammaconj} we are done after considering connected components one at a time.
\end{proof}
In light of \lemref{rebase}, when we are interested only in the rank of $A(G,B,F)$, we can freely
change $B$ and $F$.  Thus, we shorten the notation to $A(G)$.

\subsubsection{The effect on $A(G)$ of adding or deleting a colored edge}
In the proof of the basis exchange property, we will need to start with a
$\Gamma$-$(1,1)$ graph, and add a colored edge to it.  There are two possibilities:
the edge $ij$ is in the span of some connected component $G_i$ of $G$ or it is not.  Each
of these has an interpretation in terms of how $A(G+ij)$  is different from $A(G)$.
\begin{lemma}\lemlab{add-edge-AG}
Let $(G,\bgamma)$ be a colored graph and let $ij$ be a colored edge.  Then:
\begin{itemize}
\item %
If the edge $ij$ is in the span of a connected component, $G_\ell$ of $G$, then
$A(G+ij)$ is $A(G)+(\gamma,\ell)$, where $\gamma$ is the image of the fundamental closed
path of $ij$ with respect to some spanning tree and base vertex of $G_\ell$.
\item %
If the edge $ij$ connects two connected components $G_\ell$ and $G_r$ of $G$, then
$A(G+ij)$ is a fusing operation (defined in \secref{group-matroid}) on $A(G)$ after a conjugation.  In
particular, in the notation of \secref{group-matroid}, $A(G)_\ell$ and $A(G)_r$ are fused.  Conversely,
$A(G)$ is a conjugation of a separation of $A(G+ij)$.
\end{itemize}
\end{lemma}
\begin{proof}
The first part follows from the fact that if we pick a base vertex and spanning tree of $G_\ell$, then adding the
colored edge $ij$ to $G_\ell$ induces exactly one new fundamental closed path.

For the second part, w.l.o.g., assume that $G$ has two connected components $G_1, G_2$ and that $ij$ connects them.
Let $T_1$ and $T_2$ be the spanning trees and $b_1$ and $b_2$ the base vertices which define the set $A(G)$. Then,
we can choose $T = T_1 \cup T_2 + ij$ as the spanning tree for $G$ and $b_1$ for the base vertex. The fundamental
closed paths for edges in $E(G_1) - E(T)$ are unchanged. The $\rho$-image of the fundamental paths for $E(G_2) - E(T)$ are
conjugated by $\rho(P_{b_1 b_2})$ where $P_{b_1 b_2}$ is the unique path in $T$ from $b_1$ to $b_2$. Thus
$A(G+ij)_1$ consists of $A(G)_1$ and a conjugation of $A(G)_2$. The converse is clear since the inverse of a conjugation is a conjugation, and the inverse
of fusing is separating.
\end{proof}

\subsubsection{$\Gamma$-$(1,1)$ graphs and tight independent sets in $M_{\Gamma_k,n}$}
$\Gamma$-$(1,1)$ graphs $(G,\bgamma)$ have a simple characterization in terms of $A(G)$: they
correspond exactly to the situations in which $A(G)$ is tight and independent.

\begin{lemma}\lemlab{gamma11-tight}
Let $(G,\bgamma)$ be a colored graph.  Then $(G,\bgamma)$ is $\Gamma$-$(1,1)$ if and only if $A(G)$
is tight and independent in $M_{\Gamma_k,n}$.
\end{lemma}

\begin{proof}
We recall that \lemref{Mgammatight} gives a structural characterization of tight independent sets in $M_{\Gamma_k,n}$.
The proof proceeds by translating the definitions from \secref{group-matroid-sets}
into graph theoretic terms.  In this proof, we adopt the notation of
\secref{group-matroid-sets}, and we remind the reader that a subset $A\subset E_{\Gamma_k,n}$ is tight if it is
independent in $M_{\Gamma_k,n}$ and has
\[
|A| = c(A) + \frac{1}{2}\rep_{\Gamma_k}(\Trans(\Gamma_k))
\]
elements.

We first suppose that $A(G)$ is tight, and show that $(G,\bgamma)$ is a $\Gamma$-$(1,1)$ graph.
By construction $A(G)$ has an element $(\gamma,\ell)$ if and only if there is some edge $ij$ in the
connected component $G_\ell$ not in the spanning forest $F$ used to compute $A(G)$.  It then follows that,
if $A(G)$ is tight, each connected component of $G_i$ of $G$ has at least one more edge than
$G_i\cap F$.     %
This implies that $G$ contains a spanning map graph.  Because $|A(G)| = c(A) + \frac{1}{2}\rep_{\Gamma_k}(\Trans(\Gamma_k))$
it then follows that $G$ is a map-graph plus $\frac{1}{2}\rep_{\Gamma_k}(\Trans(\Gamma_k))$ additional edges,
which are the combinatorial hypotheses for being a $\Gamma$-$(1,1)$ graph.

Now we use the fact that $A(G)$ is independent in $M_{\Gamma_k,n}$.  Independence implies that,
if non-empty, $A(G_i)$ contains a rotation, from which it follows that, for each connected component $G_i$ of $G$,
$\rho(\pi_1(G_i,b_i))$ does as well.  Similarly, independence of $A(G)$ implies that
$\rep_{\Gamma_k}(\Trans(A(G))) = \rep_{\Gamma_k}(\Trans(\Gamma_k))$, so $\rep_{\Gamma_k}(G)= \rep_{\Gamma_k}(\Trans(\Gamma_k))$.
We have now shown that $(G,\bgamma)$ is a $\Gamma$-$(1,1)$ graph.

The other direction is straightforward to check.
\end{proof}

\subsubsection{$\Gamma$-$(1,1)$ graphs form a matroid}
We now have the tools to prove that the $\Gamma$-$(1,1)$ graphs form the bases of a matroid.  We take as the ground
set the graph $K_{\Gamma_k,n}$ on $n$ vertices that has one copy of each possible directed edge $ij$
or self-loop $ij$ with color $\gamma\in \Gamma_k$.

\begin{lemma}\lemlab{gamma11-bases}
The set of $\Gamma$-$(1,1)$ graphs on $n$ vertices
form the bases of a matroid on $K_{\Gamma_k,n}$.
\end{lemma}
\begin{proof}
We check the basis axioms \cite[Section 1.2]{O11}:

\noindent
\textbf{Non-triviality:}
An uncolored tree plus $\frac{1}{2}\rep_{\Gamma_k}(\Gamma_k) + 1$ edges,
each of which is colored by a standard generator for $\Gamma_k$ is clearly \GammaOO.  Thus the set of bases is not empty.

\noindent
\textbf{Equal size:}
By definition, all \GammaOO{} graphs have the same number of edges.

\noindent
\textbf{Basis exchange:}
The more difficult step is checking basis exchange.  To do this we let $G$ be a \GammaOO{} graph and $ij$ a colored edge
of some other \GammaOO{} graph which is not in $G$.  It is sufficient to check that there is some colored edge $i'j'\in E(G)$
such that $G+ij-i'j'$ is also a \GammaOO{} graph.  Let $(G',\bgamma')$ be the colored graph $(G+ij,\bgamma)$.

Suppose the new edge $ij$ is not a self-loop. Then,
pick base vertices $B$ and a spanning forest $F$ of $G'$ that contains the new edge $ij$.
By \lemref{rebase}, changing $F$ so as to include $ij$ does not change the rank of $A(G',B,F)$ in
$M_{\Gamma_k,n}$.  \lemref{add-edge-AG} implies that $A(G',B,F)$ is spanning, but not
independent, in $M_{\Gamma_k,n}$.  Thus there is an element of
$A(G',B,F)$ that can be removed to leave a tight, independent set.  Since $ij$ is in $F$, this
element does not correspond to $ij$.  The basis exchange axiom then follows from the characterization
of $\Gamma$-$(1,1)$ graphs in \lemref{gamma11-tight}.

Suppose $ij$ is a self-loop. Then, the first conclusion of \lemref{add-edge-AG} applies.
Since $ij$ comes from some other \GammaOO graph, it has non-trivial color, and, thus is not dependent
as a singleton set. It follows that there is some element in $A(G')$ (not corresponding to $ij$) which can
be removed to give a tight independent set of the matroid $M_{\Gamma_k, n}$. Consequently, removing the
corresponding edge in $G'$ leaves a \GammaOO graph by \lemref{gamma11-tight}.
\end{proof}

\subsubsection{The rank function of the $\Gamma$-$(1,1)$ matroid}
Now we compute the rank function of the $\Gamma$-$(1,1)$ matroid.  The
following lemma is immediate from the definitions.
\begin{lemma}\lemlab{g-and-g1}
Let $(G,\bgamma)$ be a colored graph with $n$ vertices and $c$ connected components.
Then
\[
g(G) = n - c + g_1(A(G))
\]
where $g_1$ is the rank function of the matroid $M_{\Gamma_k,n}$.
\end{lemma}
We can use this to show:
\begin{lemma}\lemlab{gamma11-independent-g}
Let $(G,\bgamma)$ be a colored graph that is independent in the $\Gamma$-$(1,1)$ matroid with $m$ edges.
Then $m=g(G)$.
\end{lemma}
\begin{proof}
By definition $(G, \bgamma)$ is a subgraph of some \GammaOO{} graph $(G', \bgamma')$.  By \lemref{gamma11-tight},
$m' = g(G')$, where $m'$ is the number of edges of $G'$.  It suffices to show that deleting an edge preserves this equality
and independence of $A(G')$.  By \lemref{add-edge-AG}, deleting an edge is equivalent to either removing an element from
$A(G')$ or separating and conjugating $A(G')$ and these both preserve independence
of $A(G')$.  In the first case, $g_1(A(\cdot))$ drops by $1$ while $n'$ and $c'$ remain constant, and in the second
case $n'$ and $g_1(A(\cdot))$ remain constant while $c'$ increases by $1$.
\end{proof}
We can now compute the rank function of the $\Gamma$-$(1,1)$ matroid.
\begin{lemma}\lemlab{gamma11-rank}
The function $g$ is the rank function of the $\Gamma$-$(1,1)$ matroid.
\end{lemma}
\begin{proof}
Let $(G,\bgamma)$ be an arbitrary colored graph with $n$ vertices and $c$ connected components.  The rank of
$(G,\bgamma)$ in the $\Gamma$-$(1,1)$ matroid is equal to the maximum size of the
intersection of $G$ with a $\Gamma$-$(1,1)$ graph.  \lemref{gamma11-independent-g} implies that what
we need to show is that a maximal independent subgraph $(G',\bgamma)$ of $(G,\bgamma)$ has $g(G)$ edges.

We construct $G'$ as follows.  First pick a base vertex for every connected component of $G$ and a
spanning forest $F$ of $G$.  Initially set $G'$ to be $F$.  Then add edges one at a time
to $G'$ from $G-F$ so that $A(G')$ remains independent in $M_{\Gamma_k,n}$
until the rank of $A(G')$ is equal to that of $A(G)$.  This is possible by the
matroidal property of $M_{\Gamma_k,n}$ and \lemref{rebase}, which says the rank of $A(G')$ is
invariant under the choices of spanning forest and base vertices.

When the process stops, $A(G')$ is independent in $M_{\Gamma_k,n}$, so $G'$ is
independent in the $\Gamma$-$(1,1)$ matroid by \lemref{gamma11-tight}.  By construction $G'$ has
\[
m' = n - c + g_1(A(G))
\]
edges, which is $g(G)$ by \lemref{g-and-g1}.
\end{proof}

\subsection{Cone-$(1,1)$ and cone-$(2,2)$ graphs} \seclab{cone11-22}
In the proof of \theoref{direction} (specifically, \secref{gen-cone22-direction-networks} below),
we will require some results on direction networks with rotational symmetry
from \cite{MT11,MT12fsi}.  The combinatorial setup is given in this
short section.
\subsubsection{Cone-$(1,1)$ graphs}
Let $(G, \bgamma)$ be a graph whose edges are colored by elements of the group $\Z/k\Z$. As before, there
is a well-defined map $\rho: \pi_1(G_i, b_i) \to \Z/k\Z$ where $b_i$ is a vertex in the connected component $G_i$ of $G$.
We define $(G,\bgamma)$ to be a cone-$(1,1)$ graph if $G$ is a map-graph and the cycle in each connected component
has non-trivial $\rho$-image.
We define the quantity $T(G_i)$ to be the same one defined in \secref{gamma22}, where all nontrivial elements
of $\Z/k\Z$ are ``rotations''.

The sparsity characterization of cone-$(1,1)$ graphs is:
\begin{lemma}[name={\cite[Section 2.6]{MT12fsi}, \cite[``Matroid Theorem'']{Z82}}]
\lemlab{cone-11-sparse}
The cone-$(1,1)$ graphs on $n$ vertices are the bases of a matroid that has as its rank function
\[
r(G') = n' - \frac{1}{2}\sum_{i=1}^c T(G_i)
\]
where $n'$ and $c'$ are the number of vertices and connected components in $G'$.
\end{lemma}

\subsubsection{Cone-$(2,2)$ graphs}
Let $(G,\bgamma)$ be a $\Z/k\Z$ colored graph with $n$ vertices.  We define $(G,\bgamma)$
to be a cone-$(2,2)$ graph if:
\begin{itemize}
\item $G$ has $m = 2n$ edges.
\item For all subgraphs with $m'$ edges, $n'$ vertices, and connected components $G_1,G_2,\ldots, G_c$,
\[
m' \le 2n' - \sum_{i=1}^c T(G_i)
\]
\end{itemize}
If only the second condition holds, then $(G,\bgamma)$ is defined to be
\emph{cone-$(2,2)$ sparse}.

\subsection{Generalized cone-$(2,2)$ graphs}\seclab{gen-cone-sparse}
As a technical tool in the proof of \theoref{direction}, we will use
\emph{generalized cone-$(2,2)$ graphs}.  These are $\Gamma_k$-colored
graphs, which we will define in terms of a decomposition property.

\subsubsection{Generalized cone-$(1,1)$ graphs}
Let $(G,\bgamma)$ be a $\Gamma_k$-colored graph.  We define $(G,\bgamma)$ to be a
\emph{generalized cone-$(1,1)$ graph} if, after considering the $\rho$-image modulo the translation subgroup,
the result is a cone-$(1,1)$ graph.  Equivalently, $(G,\bgamma)$ is a generalized cone-$(1,1)$ graph if:
\begin{itemize}
\item $G$ is a map graph
\item The $\rho$-image of the cycle in each connected component of $G$ is a rotation
\end{itemize}
The difference between cone-$(1,1)$ graphs and generalized cone-$(1,1)$ graphs is that
the rotations need not be around the same center. By modding colors out by $\Trans(\Gamma_k)$, the next lemma
follows easily from \lemref{cone-11-sparse}.
\begin{lemma}\lemlab{gen-cone-11-sparse}
The generalized cone-$(1,1)$ graphs on $n$ vertices are the bases of a matroid that has as its rank function
\[
r(G') = n' - \frac{1}{2}\sum_{i=1}^c T(G_i)
\]
where $n'$ and $c'$ are the number of vertices and connected components in $G'$.
\end{lemma}

\subsubsection{Relation to $\Gamma$-$(1,1)$ graphs}
Generalized cone-$(1,1)$ graphs are related to $\Gamma$-$(1,1)$ graphs by this next sequence of
lemmas.
\begin{lemma}\lemlab{gamma11-is-gc11-spanning}
Let $(G,\bgamma)$ be a $\Gamma$-$(1,1)$ graph.  Then $(G,\bgamma)$ contains a
generalized cone-$(1,1)$ graph as a spanning subgraph.
\end{lemma}
\begin{proof}
This follows from the definition, since each connected component $G_i$ of $G$ has $T(G_i)=0$.  It follows
that $G_i$ has a spanning subgraph that is a connected map-graph with its cycle having a rotation as its
$\rho$-image.
\end{proof}

Let $(G,\bgamma)$ be a $\Gamma$-$(1,1)$ graph, and let $(G',\bgamma)$ be a spanning generalized cone-$(1,1)$
subgraph.  One exists by \lemref{gamma11-is-gc11-spanning}. We define $(G',\bgamma)$ to be a \emph{g.c.-$(1,1)$ basis}
of $(G,\bgamma)$.
\begin{lemma}\lemlab{gc11-circuits}
Let $(G,\bgamma)$ be a $\Gamma_k$-colored $\Gamma$-$(1,1)$ graph for $k=3,4,6$.  Let $(G',\bgamma)$
be a g.c.-$(1,1)$ basis of $(G,\bgamma)$, and let $ij$ be the (unique) edge in $E(G)-E(G')$.  Then either:
\begin{itemize}
\item The colored edge $ij$ is a self-loop and the color $\gamma_{ij}$ is a translation.
\item There is a unique minimal subgraph $G''$ of $G$, such that the $\rho$-image of $(G'',\bgamma)$
includes a translation, $ij$ is an edge of $G''$, and if $vw\in E(G'')$, then $(G'+ij-vw,\bgamma)$
is also a g.c.-$(1,1)$ basis of $(G,\bgamma)$.
\end{itemize}
\end{lemma}
\begin{proof}
If $ij$ is a self-loop colored by a translation, then it is a circuit in the matroid of
generalized cone-$(1,1)$ graphs on the ground set $(G,\bgamma)$.  Otherwise, the subgraph $G''$
which the lemma requires is just the fundamental generalized cone-$(1,1)$ circuit of
$ij$ in $(G',\bgamma)$.
\end{proof}

\subsubsection{Generalized cone-$(2,2)$ graphs}
Let $(G,\bgamma)$ be a $\Gamma_k$-colored graph.  We define $(G,\bgamma)$ to be a
\emph{generalized cone-$(2,2)$ graph} if it is the union of two generalized cone-$(1,1)$
graphs.  Using the Edmonds-Rota construction \cite{ER66}, the same way we did in \secref{cone22-decomp},
we get:
\begin{lemma}\lemlab{gen-cone22-matroid}
The generalized cone-$(2,2)$ graphs on $n$ vertices give the bases of a matroid.
\end{lemma}

The other fact about generalized cone-$(2,2)$ graphs is their relationship to
$\Gamma$-$(2,2)$ graphs.
\begin{lemma}\lemlab{gamma22-is-cone22-spanning}
Let $(G,\bgamma)$ be a $\Gamma$-$(2,2)$ graph.  Then $(G,\bgamma)$ contains a
generalized cone-$(2,2)$ graph as a spanning subgraph.
\end{lemma}
\begin{proof}
This follows from \propref{gamma22-decomp} and \lemref{gamma11-is-gc11-spanning}.
\end{proof}

\section{Direction networks}\seclab{bigsec-dn}
\subsection{Crystallographic direction networks} \seclab{cdns}
Let $(\tilde{G},\varphi)$ be a graph with a $\Gamma_k$-action $\varphi$.  A \emph{crystallographic direction network}
$(\tilde{G},\varphi,\tilde{\vec d})$ is given by $(\tilde{G},\varphi)$ and an assignment of
a direction $\tilde{\vec d}_{ij}\in\R^2\setminus \{0\}$ to each edge $ij\in E(\tilde{G})$.

We will, moreover, require that the direction networks themselves be symmetric in the following sense. For any
$\Phi \in \Rep(\Gamma)$ and any $\gamma \in \Gamma$, the rotational part $\Phi(\gamma)_r$ of $\Phi(\gamma)$
depends only on $\gamma$ and not $\Phi$. Thus, we will require the directions to be equivariant
with respect to this action; i.e., if
$i'j' = \gamma \cdot ij,$ then $\vec d_{i'j'} = \Phi(\gamma)_r \vec d_{ij}$.

A direction network on $(\tilde{G}, \varphi)$, thus, is completely determined after assigning a direction to one edge in each
$\Gamma$-orbit. It is plain, then, that this is equivalent to assigning directions to the colored quotient graph $(G, \bgamma)$.
The dictionary is straightforward as well. Given $(\tilde{G}, \varphi, \vec d)$, the direction for edge $ij$ of $G$ is the same
as the direction of the chosen edge representative in its fiber.

\subsubsection{The realization problem}
A \emph{realization} of a crystallographic direction network is given by a point set $\vec p=(\vec p_i)_{i\in V(\tilde{G})}$
and a representation $\Phi \in \overline{\Rep}(\Gamma_k)$:
\begin{eqnarray}
\iprod{\vec p_j - \vec p_i}{\tilde{\vec d}^\perp_{ij}} = 0 & \text{for all edges $ij\in E(\tilde{G})$} \\
\vec p_{\gamma\cdot i} = \Phi(\gamma)\cdot \vec p_i & \text{for all vertices $i\in V(\tilde{G})$}
\end{eqnarray}
We denote realizations by $\tilde{G}(\vec p,\Phi)$, to indicate the dependence on $\Phi$.

We define now \emph{collapsed} and \emph{faithful} realizations.
An edge $ij$ is \emph{collapsed} in a realization $\tilde{G}(\vec p,\Phi)$ if $\vec p_i = \vec p_j$.
A \emph{realization is collapsed} when all the edges are collapsed \emph{and} $\Phi$ is trivial.  A representation
is \emph{trivial} if it maps $\Lambda(\Gamma_k)$ to zero.  A realization
is \emph{faithful} if no edge is collapsed and $\Phi$ is not trivial.

\subsubsection{Direction Network Theorem}
Our main theorem on crystallographic direction networks is the following.
\directionthm*

For technical simplification, we reduce \theoref{direction} to the following proposition
which is the same result except that the rotation center of $\Phi(r_k)$ is fixed to be the origin.
\begin{prop} \proplab{dnthmfixedcenter}
Let $\Gamma$ be an orientation-preserving crystallographic group.  A
generic $\Gamma$-crystallographic direction network $\Gad$ has a unique, up to scaling,
faithful realization $\tilde G(\vec p, \Phi)$ satisfying $\Phi(r_k) = ( 0, R_k)$
if and only if its associated colored graph is $\Gamma$-colored-Laman.
\end{prop}

\theoref{direction} follows easily from the proposition.
\begin{proof}[Proof of \theoref{direction} from \propref{dnthmfixedcenter}]
Let $\Gad$ be any generic direction network. For any realization $\tilde G(\vec p, \Phi)$
of $\Gad$ and any translation $\psi$, we have that $\tilde G(\psi(\vec p), \Phi^\psi)$
is also realization of $\Gad$ where $\Phi^\psi$ is the representation defined by
$\Phi^\psi(\gamma) = \psi \Phi(\gamma) \psi^{-1}$. In particular, for any realization
$\tilde G(\vec p, \Phi)$, there is a unique translation $\psi$ such that $\psi \Phi(r_k) \psi^{-1} = (0, R_k)$.
\end{proof}

Our goal then is to prove \propref{dnthmfixedcenter}. For the remainder of this section, we will
require all realizations to
map the rotational generator $r_k$ of $\Gamma_k$ to the counter-clockwise rotation around the origin $R_k$ of angle $2\pi/k$.
All of the propositions in the remainder of this section operate under this assumption.
However, note that some of the results about direction networks, including
\propref{gamma-laman-circuit-collapse}, hold without this restriction.

\subsubsection{Proof of \propref{dnthmfixedcenter}}
Let $(G,\bgamma)$ be a colored graph. The key proposition, which is proved in
\secref{crystal-collapse} is the following.
\begin{restatable}{prop}{crystalcollapseprop} \proplab{crystal-collapse}
A generic crystallographic direction network which has a $\Gamma$-$(2,2)$ colored quotient graph
has only collapsed realizations.
\end{restatable}
By counting the dimension of the space of collapsed realizations, done in
\secref{gamma-laman-circuit-collapse} below, it follows that:
\begin{restatable}{prop}{gammalamancircuitcollapse} \proplab{gamma-laman-circuit-collapse}
A generic crystallographic direction network which has a $\Gamma$-colored-Laman circuit as its quotient graph
has only realizations with collapsed edges.
\end{restatable}
\propref{gamma-laman-circuit-collapse}
readily implies one direction of \propref{dnthmfixedcenter}. If $(G,\bgamma)$ is not a $\Gamma$-colored-Laman
graph, then it has either too few edges or contains a $\Gamma$-colored-Laman circuit as a subgraph.
In the former case, a dimension count implies that a faithful realization cannot be unique up to
translation and scale and in the latter, every realization contains collapsed edges by
\propref{gamma-laman-circuit-collapse}.

For the other direction, we assume that $(G,\bgamma)$ is $\Gamma$-colored-Laman with $n$ vertices.
Since every $\Gamma$-colored-Laman graph is $\Gamma$-$(2,2)$ sparse by \propref{doubleedge},
we see from \propref{crystal-collapse} that the equations defining the realization space
of a generic colored direction network with quotient graph $G(\bgamma)$ is $1$-dimensional.
This means every realization is a rescaling of a single realization, so if any edge $ij$ is
collapsed, it is collapsed in all realizations.  In particular, if we double $ij$ and
assign it a generic direction, the realization space will not change.
\propref{doubleedge} tells us that the graph obtained in this way is $\Gamma$-$(2,2)$,
so \propref{crystal-collapse} applies to it, showing that the realization space is
zero-dimensional.  The resulting contradiction completes the proof.\eop

\subsubsection{Colored direction networks}\seclab{colored-crystal-networks}  We will make use of
\emph{colored crystallographic direction networks} to study crystallographic direction networks.  Since
there is no chance of confusion, we simply call these ``colored direction networks''
in the next several sections.
A \emph{colored direction network} $(G,\bgamma,\vec d)$ is given by a $\Gamma_k$-colored graph $(G,\bgamma)$
and an assignment of a direction $\vec d_{ij}$ to every edge $ij$.  The realization system for
$(G,\bgamma,\vec d)$ is given by
\begin{eqnarray}
\iprod{\Phi(\gamma_{ij})\cdot\vec p_j - \vec p_i}{\vec d_{ij}^\perp} = 0
\label{colored-crystal-directions}
\end{eqnarray}
The unknowns are the representation $\Phi$ of $\Gamma_k$ and the points $\vec p_i$. (As above $\Phi(r_k)$
is restricted to be rotation about the origin.)  We denote points in the realization space by $G(\vec p,\Phi)$.
We observe here that, for $\Phi$ parameterized by the vectors in \lemref{repspace}, the realization system
is linear. This can be seen by, e.g.,  the computations in \secref{gen-cone22-direction-networks}.
The following two lemmas linking crystallographic direction networks and colored direction networks
follow easily from \lemref{lifts} and the observations above.
\begin{lemma}\lemlab{crystal-direction-network-lifts}
Given a colored direction network $(G,\bgamma,\vec d)$, its lift to a crystallographic direction network
$(\Gtilde, \varphi, \tilde{\vec d})$ is well-defined and the realization spaces of $(G,\bgamma,\vec d)$ and
$(\Gtilde, \varphi, \tilde{\vec d})$ are isomorphic.  In particular, they have the same dimension.
\end{lemma}
\begin{lemma}\lemlab{crystal-direction-network-lifts2}
Let $(G,\bgamma)$ be a $\Gamma_k$-colored graph and $(\Gtilde, \varphi)$ its lift.  Assigning a direction
to one representative of each edge orbit under $\varphi$ in $\Gtilde$ gives a well defined
colored direction network $(G,\bgamma,\vec d)$.
\end{lemma}
This next lemma, which is also immediate from the definitions, describes collapsed edges in
terms of colored direction networks.
\begin{lemma}\lemlab{colored-collapsed-edge}
Let $(G,\bgamma,\vec d)$ be a colored direction network and let $G(\vec p,\Phi)$ be a realization of $(G,\bgamma,\vec d)$.
Let $(\tilde{G},\varphi,\tilde{\vec d})$ be the lift of $(G,\bgamma,\vec d)$ and $\tilde{G}(\vec p,\Phi)$ be the associated
lift of $G(\vec p,\Phi)$.  Then a colored edge $ij\in E(G)$ lifts to an orbit of collapsed edges
in $\tilde{G}(\vec p,\Phi)$ if and only if
\[
\vec p_i = \Phi(\gamma_{ij})\cdot\vec p_j
\]
in $G(\vec p,\Phi)$.
\end{lemma}
In light of Lemmas \ref{lemma:crystal-direction-network-lifts}--\ref{lemma:colored-collapsed-edge}, we
may switch freely between the formalisms, and we do so in subsequent sections.

\subsubsection{A result on cone direction networks}
We prove \propref{crystal-collapse} by bootstrapping results for
generalized cone-$(2,2)$ graphs (defined in \secref{gen-cone-sparse}).
The steps are:
\begin{itemize}
\item We show that, for fixed $\Phi$,
a generic direction network on a generalized cone-$(2,2)$ graph
has a unique solution (\propref{gen-cone-unique}).
\item Then we allow $\Phi$ to flex. We show that by adding $\rep_{\Gamma_k}(\Trans(\Gamma_k))$ edges
that extend a generalized cone-Laman graph to a $\Gamma$-$(2,2)$ graph, realizations of a generic
direction network are forced to collapse.
\end{itemize}
This is done in Sections \ref{sec:gen-cone22-direction-networks} and \ref{sec:crystal-collapse}.  In the
first step, we will make use of a result on finite direction networks with rotational symmetry
from \cite{MT12fsi}.
A \emph{cone direction network}
$(G, \bgamma, \vec d$) is an assignment $\vec d_{ij}$ to each edge $ij$. A realization $G(\vec p)$
of the direction network is a selection of points $\vec p_i \in \R^2$ such that
\begin{equation} \label{conedn}
\langle R_k^{\gamma_{ij}} \vec p_j - \vec p_i,\vec d_{ij} \rangle = 0 \text{ for all } ij
\end{equation}
Here $R_k$ is the rotation about the origin that rotates through angle $2 \pi/k$.
A straightforward application of \cite[Proposition 3.1]{MT12fsi} yields:
\begin{theorem}[name={\cite{MT12fsi}}] \theolab{conedn} The system \eqref{conedn} defining a generic cone
direction network $(G, \bgamma, \vec d)$ is independent if and only if
$(G, \bgamma)$ is cone-$(2,2)$ sparse.
\end{theorem}

\subsection{Direction networks on generalized cone-$(2,2)$ graphs} \seclab{gen-cone22-direction-networks}
Let $(G,\bgamma)$ be a generalized cone-$(2,2)$ graph.  In light of \theoref{conedn}, it should be unsurprising
that the realization system \eqref{colored-crystal-directions} has generic rank $2n$ for a colored
direction network on $(G,\bgamma)$, since cone direction networks are a ``special case''.  Here is
the precise reduction.
\begin{prop}\proplab{gen-cone-unique}
Fix a representation $\Phi$ of $\Gamma_k$.  Holding $\Phi$ fixed, if a generic crystallographic
direction network has a generalized cone-$(2,2)$ colored quotient, then it has a unique
realization.
\end{prop}
\propref{gen-cone-unique} is immediate from the following statement and \lemref{crystal-direction-network-lifts}.
\begin{prop}\proplab{gen-cone-rank}
Let $(G,\bgamma)$ be a generalized cone-$(2,2)$ graph with $n$ vertices.
Then the generic rank of the realization system \eqref{colored-crystal-directions} is $2n$.
\end{prop}
\begin{proof}
Expanding \eqref{colored-crystal-directions} we get
\begin{equation}
\label{gen-cone-rank-1}
\iprod{\Phi(\gamma_{ij})\cdot\vec p_j - \vec p_i}{\vec d_{ij}^\perp}  =
\iprod{\Phi(\gamma_{ij})\cdot\vec p_j}{\vec d_{ij}^\perp} - \iprod{\vec p_i}{\vec d_{ij}^\perp}
\end{equation}
Define $\Phi(\gamma_{ij})_{r}\in SO(2)$ to be the rotational part of $\Phi(\gamma_{ij})$ and $\Phi(\gamma_{ij})_t\in \R^2$
to be the translational part, so that $\Phi(\gamma_{ij})\cdot \vec p = \Phi(\gamma_{ij})_{r}\cdot \vec p + \Phi(\gamma_{ij})_t$.
In this notation, \eqref{gen-cone-rank-1} becomes
\begin{equation}
\label{gen-cone-rank-2}
\iprod{\Phi(\gamma_{ij})_r\cdot\vec p_j}{\vec d_{ij}^\perp} + 	\iprod{\Phi(\gamma_{ij})_t}{\vec d_{ij}^\perp} -
\iprod{\vec p_i}{\vec d_{ij}^\perp} = 0
\end{equation}
Since the rotational part $\Phi(\gamma_{ij})_r$ preserves the inner product, we see that \eqref{colored-crystal-directions}
is equivalent to the inhomogeneous system
\begin{equation}
\label{gen-cone-rank-3}
\iprod{\vec p_j}{\Phi(\gamma_{ij}^{-1})_r\cdot\vec d_{ij}^\perp}  -
\iprod{\vec p_i}{\vec d_{ij}^\perp} = - 	\iprod{\Phi(\gamma_{ij})_t}{\vec d_{ij}^\perp}
\end{equation}
The l.h.s. of \eqref{gen-cone-rank-3} is equivalent to \eqref{conedn}, and thus the generic
rank of \eqref{gen-cone-rank-3} is at least as large as that of  \eqref{conedn}.  The proposition
then follows from \theoref{conedn}.
\end{proof}

\subsection{Proof of \propref{crystal-collapse}}\seclab{crystal-collapse}
We now have the tools in place to prove:
\crystalcollapseprop*
The proof is split into two cases, $\Gamma_2$ and $\Gamma_k$ for $k=3,4,$ or $6$.

\subsubsection{Proof for rotations of order $3$, $4$, or $6$}
Let $(G,\bgamma)$ be a $\Gamma$-$(2,2)$ graph.  We construct a direction network
on $(G,\bgamma)$ that has only collapsed solutions, from which the desired
generic statement follows.

\paragraph{Assigning directions}
We select directions $\vec d$ for each edge in $G$ with coordinates that
are algebraically independent over $\mathbb{Q}$; i.e.,
their coordinates satisfy no polynomial with integer coefficients.

\paragraph{The realization space of any spanning g.c.-$(2,2)$ basis}
With these direction assignments, we can compute the dimension
of the realization space for the direction network
induced on any spanning g.c.-$(2,2)$ basis of $(G,\bgamma)$ where by g.c.-$(2,2)$ basis
we mean a basis in the generalized cone-$(2,2)$ matroid.  One
exists by \lemref{gamma22-is-cone22-spanning}.

\begin{lemma}\lemlab{collapse-proof-dimension}
Let $(G',\bgamma)$ be a spanning g.c.-$(2,2)$ basis of $(G,\bgamma)$.  Then the
realization space of the induced direction network $(G',\bgamma,\vec d)$
is $2$-dimensional, and linearly depends on the representation $\Phi$.
\end{lemma}
\begin{proof}
The dimension comes from \propref{gen-cone-unique} and comparing the number of
variables to the number of equations in the realization system
\eqref{colored-crystal-directions}.  Moving the variables associated with $\Phi$
to the right as in equation \eqref{gen-cone-rank-3} completes the proof.
\end{proof}

\paragraph{A g.c.-$(2,2)$ basis with non-collapsed complement}
By edge counts, there are exactly two edges $ij$ and $vw$ in the
complement of any g.c.-$(2,2)$ basis of $(G, \bgamma)$. We will show that there
are two edges which do not collapse (and more) when enforcing the directions
on the complement.

\begin{lemma}\lemlab{collapse-proof-good-basis}
There are edges $ij, vw$ of $G$ such that $G' = G - ij - vw$ is a g.c.-$(2,2)$ basis and
for all vectors $\vec u \in \R^2$ there is a realisation $G(\vec p, \Phi)$ of
$(G', \bgamma, \vec d)$ such that $\Phi(\gamma_{ij}) \vec p_j - \vec p_i = \vec u$
(similarly there is a realization such that $\Phi(\gamma_{vw}) \vec p_w - \vec p_v = \vec u$).
\end{lemma}
\begin{proof}
By \propref{gamma22-decomp}, we can decompose $(G,\bgamma)$ into two spanning
$\Gamma$-$(1,1)$ graphs $X$ and $Y$. Since $\Gamma$-$(1,1)$ graphs are g.c.-$(1,1)$ graphs
plus an edge, there are unique g.c.-$(1,1)$ circuits $X''$ and $Y''$ in $X$ and $Y$ respectively.
Let $v'w'$ be some edge in $Y''$ and let $Y' = Y - v'w'$.

Suppose, for contradiction, that for an arbitrary edge $ij$ of $X''$, the vector
$\Phi(\gamma_{ij}) \vec p_j - \vec p_i$ is constrained to a one-dimensional
subspace (or smaller) over all realizations $G(\vec p, \Phi)$ of the direction network
$(X - ij \cup Y', \bgamma, \vec d)$. Then, either the vector is identically zero or by
genericity of $\vec d$, the vector $\vec d_{ij}$ differs from $\Phi(\gamma_{ij}) \vec p_j - \vec p_i$.
In either case the edge $ij$ is collapsed in all realizations of $(X \cup Y', \bgamma, \vec d)$.
Since $ij$ was arbitrary in $X''$, all edges in $X''$ are collapsed in all realizations of
$(X \cup Y', \bgamma, \vec d)$.

However, if every edge in $X''$ is collapsed
in every realization of the direction network $(X \cup Y',\bgamma, \vec d)$,
this implies that $\Phi$ must always be trivial in any realization.
\propref{gen-cone-unique} would then imply that the realization space is
$0$-dimensional, and this contradicts the fact that it is at least $1$-dimensional,
by \lemref{collapse-proof-dimension}.
Thus, it must be that for some edge $ij$
in $X''$, the vector $\Phi(\gamma_{ij}) \vec p_j - \vec p_i$ sweeps out
all of $\R^2$ as $G(\vec p, \Phi)$ varies over all realizations of
$(X - ij \cup Y', \bgamma, \vec d)$: any direction is achievable by changing $\Phi$ and we can scale.
Let now $X' = X - ij$.

By reversing the roles of $X$ and $Y$ we can find an edge $vw$ of $Y$ with the same properties.
(Note that the $ij$ we chose was in $X''$ so the situation is symmetric!)
\end{proof}

\paragraph{The representation $\Phi$ must be trivial}
The rest of the proof will be to show that, adding back $ij$ and
$vw$ forces all realizations of $(G,\bgamma,\vec d)$ to collapse.
\lemref{collapse-proof-good-basis} and \lemref{collapse-proof-dimension} tell us that
for realizations $G(\vec p, \Phi)$ of $(G', \bgamma, \vec d)$,
both vectors $\Phi(\gamma_{ij}) \vec p_j - \vec p_i$ and $\Phi(\gamma_{vw}) \vec p_w - \vec p_v$
depend linearly in a one-to-one fashion on $\Phi$ which parameterizes the realization space.

Consequently, if we add the edge $ij$ to $G'$, the new direction network must constrain
$\Phi$, and thus $\Phi(\gamma_{vw}) \vec p_w - \vec p_v$, to some one-dimensional space.
Since $\vec d$ was chosen generically, $\vec d_{vw}$ differs from this latter vector, and thus
$\Phi$ is trivial in a realization of $(G, \bgamma, \vec d)$. Since $\Phi$ is trivial, then
by \propref{gen-cone-unique} the unique realization must be the completely collapsed one.
\eop

\subsubsection{Proof for rotations of order $2$}
Let $(G,\bgamma)$ be a $\Gamma$-$(2,2)$ graph.  Again, we will assign directions so that the
resulting direction network $(G,\bgamma, \vec d)$ has only collapsed solutions.  The
proof has a slightly different structure from the $k=3,4,6$ case.  The main
geometric lemma is the following.

\begin{lemma}\lemlab{collapse-proof-G2-gam11}
Let $(X,\bgamma)$ be a $\Gamma$-$(1,1)$ graph with $\Gamma_2$ colors,
and let $(X,\bgamma,\vec d)$ be a colored direction network which assigns
the same direction $\vec v$ to every edge. Then, any
realization $X(\vec p,\Phi)$ lifts to a realization $\tilde{X}(\vec p, \Phi)$
such that every vertex lies on a single line in the direction of $\vec v$.
\end{lemma}

\paragraph{Proof that \lemref{collapse-proof-G2-gam11} implies \propref{crystal-collapse} for $\Gamma_2$}
With \lemref{collapse-proof-G2-gam11}, the Proposition follows readily:
the combinatorial \propref{gamma22-decomp} says we may decompose $(G,\bgamma)$ into two spanning
$\Gamma$-$(1,1)$ graphs, which we define to be $X$ and $Y$.  We assign
the edges of $X$ a direction $\vec v_X$ and the edges of $Y$ a linearly independent
direction $\vec v_Y$.  Applying \lemref{collapse-proof-G2-gam11}, to $X$ and $Y$ separately
shows that every vertex of a lifted realization $\tilde{G}(\vec p,\Phi)$ must lie in
two skew lines.  This is possible only when they are all at the intersection of these
lines, implying only collapsed realizations. \eop

\paragraph{Proof of \lemref{collapse-proof-G2-gam11}}
Let $(X,\bgamma)$ be a $\Gamma$-$(1,1)$ graph with $\Gamma_2$ colors, and let $(X,\bgamma, \vec d)$
be a direction network that assigns all the edges the same direction. Let $(X',\bgamma)$  be a spanning g.c.-$(1,1)$
basis of $(X,\bgamma)$; one exists by \lemref{gamma11-is-gc11-spanning}.

\begin{lemma}\lemlab{collapse-proof-G2-gc11}
Let $(X',\bgamma,\vec d)$ be a g.c.-$(1,1)$ graph, and let $\vec d$
assign the same direction $\vec v$ to every edge.  Then, in any realization of the lifted
crystallographic direction network $(\tilde{X},\varphi,\vec d)$, every vertex
and every edge lies on a line in the direction $\vec v$ through a rotation center
of $\Phi(\gamma)$ for some $\gamma \in \Gamma_k$.
\end{lemma}
\noindent {\it Proof:}
It will suffice to prove the lemma when $X'$ is connected.
Because the $\rho$-image of $X'$ contains
an order $2$ rotation $r$, for some vertex $i\in V(X')$, there is a vertex $\tilde{i}$ in the
fiber over $i$ such that
$\vec p_{\tilde{i}} - \vec p_{r\cdot \tilde{i}} = \vec p_{\tilde{i}} - \Phi(r)\cdot\vec p_{\tilde{i}}$
is in the direction $\vec v$.  Because $\Phi(r)$ is a rotation through angle $\pi$, this means that
$\vec p_{\tilde{i}}$ and $\vec p_{r\cdot \tilde{i}}$ lie on a line through the rotation center of $\Phi(r)$
in the direction $\vec v$.  Because $X'$ is connected, and edge directions (up to sign) are fixed under an
order $2$ rotation, the same is then true for every vertex in the
connected component $\tilde{X}_0'$ of the lifted realization $\tilde{X}(\vec p,\Phi)$ that contains $\vec p_{\tilde{i}}$.
The lemma then follows by considering translates of $\tilde{X}_0'$. \eop \lemref{collapse-proof-G2-gc11}

To complete the proof, we recall that the $\rho$-image of $(X,\bgamma)$ contains
two linearly independent translations $t$ and $t'$. This implies that in the lifted realization
$\tilde{X}(\vec p,\Phi)$, there is a vertex $\tilde i$ connected by a path of edges in
$\tilde X$ to $t(\tilde i)$. Since all edges have the same direction in a realization, there is some $\lambda \in \R$
such that $\lambda \vec v = \vec p_{t (\tilde i)} - \vec p_{\tilde i} = \Phi(t) \cdot \vec p_{\tilde i} - \vec p_{\tilde i}$.
Thus, $\Phi(t)$ is a translation in the direction of $\vec v$. The same argument applies
to $\Phi(t')$.

From this, it follows that the rotation center of all rotations $\Phi(r)$ must lie on single line.
\lemref{collapse-proof-G2-gc11} then applies, so we are done.
\eop

\subsection{Proof of \propref{gamma-laman-circuit-collapse}}\seclab{gamma-laman-circuit-collapse}
We now prove the proposition required for the ``Maxwell direction'' of \theoref{direction}:
\gammalamancircuitcollapse*
In the proof, we will use the following statement. (cf. \cite[Lemma 14.2]{MT13} for the case when
the $\rho$-image is a translation subgroup and $\Gamma = \Gamma_1$.)
\begin{lemma}\lemlab{collapsed-dimensions}
Let $(G,\bgamma,\vec d)$ be a colored direction network on a colored graph $(G,\bgamma)$ with connected components
$G_1,G_2,\ldots, G_c$.  Then $(G,\bgamma,\vec d)$ has at least
\[
\rep_{\Gamma_k}(\Trans(\Gamma_k)) - \rep_{\Gamma_k}(G) + \sum_{i=1}^c T(G_i)
\]
dimensions of collapsed realizations.
\end{lemma}
We defer the proof of \lemref{collapsed-dimensions} to \secref{proof-of-collapsed-dimensions} and first show how
\lemref{collapsed-dimensions} implies \propref{gamma-laman-circuit-collapse}.
Let $(G,\bgamma)$ be a $\Gamma$-colored-Laman circuit with $n$ vertices, $m$ edges,and $c$
connected components $G_1,G_2,\ldots G_c$.
By \lemref{gamma-laman-circuit-sparsity}, we have
\[
m = 2n + \rep_{\Gamma_k}(G) - \sum_{i=1}^c T(G_i)
\]
It follows from \propref{crystal-collapse} that for generic directions,
a colored direction network $(G,\bgamma,\vec d)$ has a space of
realizations with dimension
\[
2n + \rep_{\Gamma_k}(\Trans(\Gamma_k)) - m = \rep_{\Gamma_k}(\Trans(\Gamma_k)) - \rep_{\Gamma_k}(G) + \sum_{i=1}^c T(G_i).
\]
Applying \lemref{collapsed-dimensions} shows that in all of them
every edge is collapsed.
\eop

\subsubsection{Proof of \lemref{collapsed-dimensions}}\seclab{proof-of-collapsed-dimensions}
For now, assume that the colored graph $(G,\bgamma)$ is connected. Select a
base vertex $b$.

\paragraph{Representations that are trivial on $\Trans(G,b)$}
Consider $\Phi\in \overline{\Rep_{\Gamma_k}}(\Trans(\Gamma_k))$ such that
\[
\Phi(t) = ((0,0),\Id)
\]
for all translations $t\in \Trans(G,b)$.
These representations form a  $(\rep_{\Gamma_k}(\Trans(\Gamma_k)) - \rep_{\Gamma_k}(G))$-dimensional
space.

\paragraph{Collapsed realizations for a fixed representation}
Now we show that there are $T(G)$ dimensions of realizations with all edges collapsed.  We
do this with an explicit construction.  There are two cases.

\noindent
\textbf{Case 1: $T(G) = 2$.}
In this case, we know that the subgroup generated by $\rho(\pi_1(G,b))$ is a
translation subgroup.  Fix a spanning tree $T$ of $G$ and a point $\vec p_b\in \R^2$.
We will construct a realization with vertex $b$ mapped to $\vec p_b$ and all edges
collapsed.

For any pair of vertices $i$ and $j$, define $Q_{ij}$ to be the path in $T$ from $i$ to $j$
and define $\eta_{ij}$ to be $\rho(Q_{ij})$.  We then set $\vec p_i = \Phi(\eta_{bi}^{-1})\cdot \vec p_b$
for all vertices $i\in V(G)$ other than $b$.  Thus all vertex locations are determined by $\vec p_b$,
giving a $2$-dimensional space of realizations for this $\Phi$. We need to
check that all edges are collapsed.
If $ij$ is an edge of $T$ with color $\gamma_{ij}$, then we have
\[
\gamma_{ij}^{-1} = \eta_{bj}^{-1}\cdot\eta_{bi}
\]
Using this relation, we see that
\[
\vec p_j = \Phi(\eta_{bj}^{-1})\cdot \vec p_b = \Phi(\gamma_{ij}^{-1}\cdot\eta_{bi}^{-1})\cdot \vec p_b =
\Phi(\gamma_{ij}^{-1})\cdot \vec p_i
\]
so the edge $ij$ is collapsed.  If $ij$ is not an edge in $T$, then the fundamental closed path $P_{ij}$
of $ij$ relative to $T$ and $b$ follows $Q_{bi}$, crosses $ij$, and returns to $b$ along $Q_{jb}$.  This
gives us the relation
\[
\gamma_{ij} = \eta_{bi}^{-1}\cdot \rho(P_{ij})\cdot \eta_{bj}
\]
We then compute
\[
\Phi(\gamma_{ij})\cdot\vec p_j = (\Phi(\eta_{bi}^{-1})\cdot \Phi(\rho(P_{ij}))\cdot \Phi(\eta_{bj}))\cdot \vec p_j
\]
Since $\Phi$ is trivial on the $\rho$-images of fundamental closed paths, the r.h.s. simplifies to
\[
\Phi(\eta_{bi}^{-1})\cdot\Phi(\eta_{bj})\cdot\vec p_j = 	\Phi(\eta_{bi}^{-1})\cdot\vec p_b = \vec p_i
\]
and we have shown that all edges are collapsed.
\vspace{1ex}

\noindent
\textbf{Case 2: $T(G) = 0$.} We adopt the notation from Case 1.  As before, we fix a spanning tree $T$
and a representation $\Phi$ that is trivial on the translation subgroup $\Trans(G,b)$.  By \lemref{subgrpgen},
$\rho(\pi_1(G,b))$ is generated by a translation subgroup $\Gamma'<\Trans(G,b)$ and a rotation $r\in \Gamma_k$.
We set $\vec p_b$ to be on the rotation center of $\Phi(r)$ and define the rest of the $\vec p_i$
as before: $\vec p_i = \Phi(\eta_{bi}^{-1})\cdot \vec p_b$.  Observe that $\Phi(r)$ then fixes $\vec p_b$.

For edges $ij$ in the tree $T$, the argument that $ij$ is collapsed from Case 1 applies verbatim.  For non-tree
edges $ij$, a similar argument relating the fundamental closed path $P_{ij}$ to $Q_{bi}$ and $Q_{bj}$ yields
the relation
\[
\gamma_{ij} = \eta_{bi}^{-1}\cdot \rho(P_{ij})\cdot \eta_{bj}
\]
Since $\Phi$ is trivial on translations $t\in \Gamma'$, we see that for some $\ell$
\[
\Phi(\gamma_{ij}) = \Phi(\eta_{bi}^{-1})\cdot\Phi(r^\ell)\cdot\Phi(\eta_{bj})
\]
We then compute
\[
\Phi(\gamma_{ij})\vec p_j = \Phi(\eta_{bi}^{-1})\cdot\Phi(r^\ell)\cdot\Phi(\eta_{bj})\cdot \vec p_j =
\Phi(\eta_{bi}^{-1})\cdot\Phi(r^\ell)\cdot\vec p_b
\]
Because $\Phi(r^\ell)\cdot\vec p_b=\vec p_b$, the r.h.s. simplifies to $\vec p_i$, and so the edge $ij$ is
collapsed.

\paragraph{Multiple connected components}
The proof of the lemma is completed by considering connected components one at a time to remove the assumption that
$G$ is connected.
\eop

\section{Rigidity}\seclab{bigsec-rigidity}
\subsection{Crystallographic and colored frameworks}\seclab{continuous}
We now return to the setting of crystallographic frameworks, leading to the proof of \theoref{main}
in \secref{main-proof}.  The overall structure is very similar to \cite[Sections 16--18]{MT13},
but we give sufficient detail for completeness.

\subsubsection{Crystallographic frameworks}
We recall the following definition from the introduction: a \emph{crystallographic framework}
$\Gal$ is given by:
\begin{itemize}
\item An infinite graph $\tilde{G}$
\item A free action $\varphi$ on $\tilde{G}$ by a crystallographic group $\Gamma$ with finite quotient
\item An assignment of a \emph{length} $\ell_{ij}$ to each edge $ij\in E(\tilde{G})$
\end{itemize}
In what follows, $\Gamma$ will always be one of the groups $\Gamma_2$, $\Gamma_3$, $\Gamma_4$,
or $\Gamma_6$.

\subsubsection{The realization space}
A \emph{realization} $\tilde{G}(\vec p,\Phi)$ of a crystallographic framework $\Gal$ is given by
an assignment
$\vec p=\left(\vec p_{i}\right)_{i\in\Vtilde}$ of points to the vertices of $\Gtilde$ and
a representation $\Phi$ of $\Gamma \into \Euc(2)$ by Euclidean isometries
acting discretely and co-compactly, such that
\begin{align}
||\vec p_i - \vec p_j||  =  \elltilde_{ij} & \text{\qquad for all edges $ij\in \Etilde$} \label{lengthsX} \\
\Phi(\gamma)\cdot \vec p_i  =  \vec p_{\gamma(i)} &
\text{\qquad for all group elements $\gamma\in \Gamma$ and vertices $i\in\Vtilde$} \label{equivariantX}
\end{align}
We see that \eqref{equivariantX} implies that, to be realizable at all, the framework $\Gal$ must assign the same
length to each edge in every $\Gamma$-orbit of the action $\varphi$.  The condition \eqref{lengthsX} is the standard one
from rigidity theory that says the distances between endpoints of each edge realize the specified lengths.

We define the \emph{realization space} $\mathcal{R}\Gal$ (shortly $\mathcal{R}$)
of a crystallographic framework to be the set of all realizations
\[
\mathcal{R}\Gal = \left\{ (\vec p,\Phi) : \text{$\tilde{G}(\vec p,\Phi)$ is a realization of $\Gal$}
\right\}
\]

\subsubsection{The configuration space}
The group $\Euc(2)$ of Euclidean isometries acts naturally on the realization space.  Let $\psi\in \Euc(2)$ be an
isometry. For any point $(\vec p,\Phi)\in \mathcal{R}$,
\[
(\psi\circ \vec p, \Phi^\psi)
\]
is a point in $\mathcal{R}$ as well where $\Phi^\psi$ is the representation defined by
$$\Phi^\psi(\gamma) = \psi \Phi(\gamma) \psi^{-1}.$$
We define the \emph{configuration space} $\mathcal{C}\Gal$
(shortly $\mathcal{C}$) of a crystallographic framework to be the quotient $\mathcal{R}/\Euc(2)$ of
the realization space by Euclidean isometries.

Since the spaces $\mathcal{R}$ and $\mathcal{C}$ are subsets of an infinite-dimensional space, there are
some technical details to check that we omit in the interest of brevity.
Interested readers can find a development for the periodic setting
in \cite[Appendix A]{MT10a}\footnote{The reference \cite{MT10a}
is an earlier version of \cite{MT13}.}.  The present crystallographic case proceeds along the
same lines.

\subsubsection{Rigidity and flexibility}
A realization $\tilde{G}(\vec p,\Phi)$ is defined to be (continuously) \emph{rigid} if it is isolated in the configuration
space $\mathcal{C}$.  Otherwise it is \emph{flexible}.  As the definition makes clear, rigidity is a \emph{local} property
that depends on a realization.
A framework that is rigid, but ceases to be so if any orbit of bars is removes is defined to be \emph{minimally rigid}.

\subsubsection{Colored crystallographic frameworks}
In principle, the realization and configuration spaces $\mathcal{R}\Gal$ and $\mathcal{C}\Gal$ of crystallographic
frameworks could be complicated infinite dimensional objects. They are, in fact,
equivalent to the finite-dimensional configuration spaces of \emph{colored crystallographic frameworks}, which will
be technically simpler to work with. (See \propref{colored-and-crystallographic-frameworks} below.)

A \emph{colored crystallographic framework} (shortly a \emph{colored framework}) is a triple $(G,\bgamma,\bm{\ell})$,
where $(G,\bgamma)$ is a $\Gamma_k$-colored graph and $\bm{\ell}=(\ell_{ij})_{ij\in E(G)}$ is an assignment of a length
to each edge. There is a dictionary between crystallographic and
colored frameworks, which is a simple modification of the dictionary for direction networks.

\subsubsection{The colored realization and configuration spaces}
A \emph{realization} $G(\vec p,\Phi)$ of a colored framework is an assignment of points
$\vec p = (\vec p_i)_{i\in V(G)}$ and a representation $\Phi$ of $\Gamma_k$ by Euclidean isometries acting discretely and
cocompactly such that
\[
||\Phi(\gamma_{ij})\cdot\vec p_j - \vec p_i||^2 = \ell_{ij}^2
\]
for all edges $ij\in E(G)$.  The \emph{realization space} $\mathcal{R}(G,\bgamma,\ell)$ is then defined to be
\[
\mathcal{R}(G,\bgamma,\ell) = \left\{ (\vec p,\Phi) :
\text{$G(\vec p,\Phi)$ is a realization of $(G,\bgamma,\bm{\ell})$}
\right\}
\]
The Euclidean group $\Euc(2)$ acts naturally on $\mathcal{R}(G,\bgamma,\ell)$ by
\[
\psi\cdot(\vec p,\Phi) = (\psi\cdot\vec p,\Phi^\psi)
\]
where $\psi$ is a Euclidean isometry.  Thus we define the
\emph{configuration space}  $\mathcal{C}(G,\bgamma,\ell)$ to be the quotient
$\mathcal{R}(G,\bgamma,\ell)/\Euc(2)$ of the realization space by the Euclidean group.

\subsubsection{The modified configuration space}
Because it is technically simpler, we will consider the modified realization space
$\mathcal{R'}(G,\bgamma,\ell)$, which we define to be:
\[
\mathcal{R'}(G,\bgamma,\ell) = \left\{ (\vec p,\Phi) :
\text{$G(\vec p,\Phi)$ is a realization of $(G,\bgamma,\bm{\ell})$ with $\Phi(r_k)$ fixing the origin}
\right\}
\]
Recall that $r_k$ is the rotation of order $k$ that is one of the generators of $\Gamma_k$.  The modified
configuration space $\mathcal{C'}(G,\bgamma,\ell)$ is then defined to be the quotient $\mathcal{R'}(G,\bgamma,\ell)/O(2)$
of the modified realization space by the orthogonal group $O(2)$.  Since every representation $\Phi\in \Rep(\Gamma_k)$
is conjugate by a Euclidean translation to a representation $\Phi'$ that has the origin as a rotation center,
this next lemma follows immediately.
\begin{lemma}\lemlab{modified-config-space}
Let $(G,\bgamma,\bm{\ell})$ be a colored framework.  Then the configuration space $\mathcal{C}(G,\bgamma,\ell)$
is homeomorphic to the modified configuration space $\mathcal{C'}(G,\bgamma,\ell)$.
\end{lemma}

From the definition and \lemref{repspace} we see that the modified configuration space is an
algebraic subset of $\R^{2n}\times \R^4$, for $\Gamma_2$ and of $\R^{2n}\times \R^2$ for
$\Gamma_k$ with $k=3,4,6$.

\subsubsection{Colored rigidity and flexibility}
We now can define rigidity and flexibility in terms of colored frameworks.  A realization $G(\vec p,\Phi)$ of a colored
framework is \emph{rigid} if it is isolated in the configuration space and otherwise \emph{flexible}.  \lemref{modified-config-space}
implies that a realization is rigid if and only if it is isolated in the modified configuration space.

\subsubsection{Equivalence of crystallographic and colored rigidity}
The connection between the rigidity of crystallographic and colored frameworks
is captured in the following proposition, which says that we can switch between the two models.
\begin{prop}\proplab{colored-and-crystallographic-frameworks}
Let $\Gal$ be a crystallographic framework and let $(G,\bgamma,\bm{\ell})$ be an associated
colored framework quotient.  Then the configuration spaces $\mathcal{C}\Gal$ and
$\mathcal{C'}(G,\bgamma,\ell)$ are homeomorphic.
\end{prop}
\begin{proof}
This follows from the definitions, a straightforward computation, and \lemref{modified-config-space}. %
\end{proof}

\subsection{Infinitesimal and generic rigidity}\seclab{infinitesimal}
As discussed above, the modified realization space $\mathcal{R}'(G,\bgamma,\bm{\ell})$ of a colored
framework is an algebraic subset of $\mathbb{R}^{2n+2r}$, where $r = \rep_{\Gamma_k}(\Gamma_k)$.
The coordinates are given as follows:
\begin{itemize}
\item The first $2n$ coordinates are the coordinates of the points $\vec p_1, \vec p_2,\ldots, \vec p_n$
\item The final $2r$ coordinates are the vectors $v_i$ specifying the representation of the translation
subgroup $\Trans(\Gamma_k)$.  (Since we have ``pinned'' a rotation center to the origin, the vector
$w$ from \lemref{repspace} is fixed to be $0$.)
\end{itemize}

\subsubsection{Infinitesimal rigidity}
As is typical in the derivation of Laman-type theorems, we
linearize the problem by considering the tangent space of $\mathcal{R}'(G,\bgamma,\bm{\ell})$
near a realization $G(\vec p,\Phi)$.

The vectors in the tangent space are infinitesimal motions of the framework, and
they can be characterized as follows.  Let
$(\vec q, \vec u_1, \vec u_2) \in \R^{2n+4}$ for $k =2$ or $(\vec q, \vec u_1) \in \R^{2n+2}$ for $k =3,4,6$.
To this vector there is an associated representation $\Phi'$ defined by
$\Phi'(r_k) = (0, R_k)$ and $\Phi'(t_i) = (\vec u_i, \Id)$.
Then differentiation of the length equations yield this linear system ranging over all
edges $ij \in E(G)$:
\begin{equation}
\iprod{\Phi(\gamma_{ij}) \cdot \vec p_j - \vec p_i}{\Phi'(\gamma_{ij}) \cdot \vec q_j  -  \vec q_i}  \label{eq:infinitesimal}
\end{equation}
The given data are the $\vec p_i$ and $\Phi$, and then unknowns are the $\vec q_i$ and $\Phi'$.
A realization $G(\vec p, \Phi)$ of a colored framework
is defined to be  \emph{infinitesimally rigid} if the  system \eqref{eq:infinitesimal}
has a $1$-dimensional solution space.  A realization that is infinitesimally rigid but ceases to be
so when any colored edge is removed is minimally infinitesimally rigid.

\subsubsection{Infinitesimal rigidity implies rigidity}
A standard kind of result relating infinitesimal rigidity and rigidity for generic frameworks holds in our setting.
Since our realization space
is finite dimensional, we can adapt the arguments of, e.g., \cite{AR78} to our setting to show:
\begin{lemma}\lemlab{infinitesimal-rigidity-implies-rigidity}
If a realization $G(\vec p,\Phi)$ of a colored framework is infinitesimally rigid, then it is rigid.
\end{lemma}

\subsubsection{Generic rigidity}
The converse of \lemref{infinitesimal-rigidity-implies-rigidity} does not hold in general, but it does for
nearly all realizations.  Let $(G,\bgamma,\bm{\ell})$ be a colored framework.  A realization
$G(\vec p,\Phi)$ is defined to be \emph{regular} for $(G,\bgamma,\bm{\ell})$
if the rank of the system \eqref{eq:infinitesimal} is maximal over all realizations.

Whether a realization is regular depends on both the colored graph $(G,\bgamma)$ and the given
lengths $\bm{\ell}$.  Let $G(\vec p,\Phi)$ be a regular realization of a colored framework.
If, in addition, the rank of \eqref{eq:infinitesimal} at $G(\vec p,\Phi)$ is maximal over all
realizations of colored frameworks with the same colored graph $(G,\bgamma)$, we define
$G(\vec p,\Phi)$ to be \emph{generic}.
We define the rank of \eqref{eq:infinitesimal} at a generic realization to be its \emph{generic rank}.
Since it depends on formal minors of the matrix underlying \eqref{eq:infinitesimal} only, it is a
property of the colored graph $(G,\bgamma)$.

If $(G,\bgamma,\bm{\ell})$ is a framework with generic realizations, it is immediate that the set
of non-generic realizations is a proper algebraic subset of the realization space.  Alternatively,
if we consider frameworks as being induced by realizations, the set of non-generic realizations
is a proper algebraic subset of $\R^{2n+2r}$, where $r=1$ for $\Gamma_{3}$, $\Gamma_{4}$, and
$\Gamma_{6}$, and $r=2$ for $\Gamma_2$.

For generic realizations, a standard argument (again, along the lines of \cite{AR78}) shows that
rigidity and infinitesimal rigidity coincide.
\begin{prop}\proplab{generic-rigidity}
A generic realization of a colored framework $(G,\bgamma,\bm{\ell})$ is rigid if and only if
it is infinitesimally rigid.
\end{prop}

\subsection{Proof of \theoref{main}}\seclab{main-proof}
We recall, from the introduction, our main theorem:
\mainthm*
The proof occupies the rest of this section.

\subsubsection{Reduction to colored frameworks}
By \propref{colored-and-crystallographic-frameworks}, it is sufficient to prove the statement
of \theoref{main} for colored frameworks.  \propref{generic-rigidity} then implies that the
theorem will follow from a characterization of generic infinitesimal rigidity for colored
frameworks.

Thus, to prove the theorem, we show that, for a colored graph $(G,\bgamma)$ with
$n$ vertices and $m=2n + \rep_{\Gamma_k}(\Trans(\Gamma_k)) - 1$ edges,
the generic rank of the system \eqref{eq:infinitesimal} is $m$ if and only if
$(G,\bgamma)$ is a $\Gamma$-colored-Laman graph.

\subsubsection{Necessity: the ``Maxwell direction''}
We recall the definition of the sparsity function $h(G)$ from \secref{gamma-laman},
which defines $\Gamma$-colored-Laman graphs.  We have, for a colored graph $(G,\bgamma)$
with $n$ vertices and $c$ connected components $G_1,G_2,\ldots,G_c$,
\[
h(G) = 2n + \rep_{\Gamma_k}(G) - 1 - \sum_{i=1}^c T(G_i)
\]
\begin{prop}\proplab{maxwell-direction}
Let $(G,\bgamma)$ be a colored graph.  Then the generic rank of the system \eqref{eq:infinitesimal} is at most
$h(G)$.
\end{prop}
\begin{proof}
Let $G(\vec p,\Phi)$ be any realization of a colored framework on a colored graph $(G,\bgamma)$ with
no collapsed edges.  That is select a representation $\Phi$ of $\Gamma_k$ and points $\vec p_i$, such
that, $\Phi(\gamma_{ij})\cdot\vec p_j\neq \vec p_i$ for all edges $ij\in E(G)$.

We now define the direction $\vec d_{ij}$ to be $(\Phi(\gamma_{ij})\cdot\vec p_j - \vec p_i)^{\perp}$ for
each edge $ij\in E(G)$.  These directions define a colored direction network $(G,\bgamma,\vec d)$ with the
property that any solution to this direction network corresponds to an infinitesimal motion of the
colored framework realization $G(\vec p,\Phi)$.

\lemref{collapsed-dimensions} implies that there are
\[
\rep_{\Gamma_k}(\Trans(\Gamma_k)) - \rep_{\Gamma_k}(G) + \sum_{i=1}^c T(G_i)
\]
dimensions of  realizations with every edge collapsed.  By construction, there is a non-collapsed realization of this
direction network as well: it is simply $(\vec p,\Phi)$ rotated by $\pi/2$.  Since this is not obtained by taking linear combinations of
realizations where every edge is collapsed, the dimension of the space of infinitesimal motions is always at
least
\[
\rep_{\Gamma_k}(\Trans(\Gamma_k)) - \rep_{\Gamma_k}(G) + \sum_{i=1}^c T(G_i) + 1
\]
The proposition follows by subtracting from $2n + \rep_{\Gamma_k}(\Trans(\Gamma_k))$ and comparing to $h(G)$.
\end{proof}

\subsubsection{Sufficiency: the ``Laman direction''}
The other direction of the proof of \theoref{main} is this next proposition
\begin{prop}\proplab{laman-direction}
Let $(G,\bgamma)$ be a $\Gamma$-colored-Laman graph.  Then the generic rank of
the system \eqref{eq:infinitesimal} is $h(G)$.
\end{prop}
\begin{proof}
It is sufficient to construct a single example at which this rank is attained, since the generic rank is
always at least the rank for any specific realization.  We will do this using direction networks.

Let $(G,\bgamma)$ be a $\Gamma$-colored-Laman graph, and select a direction $\vec d_{ij}$ for each edge
$ij\in E(G)$, such that both $\vec d$ and $\vec d^{\perp}=(\vec d_{ij}^\perp)$ are
generic in the sense of \propref{dnthmfixedcenter}.
By \propref{dnthmfixedcenter}, the colored direction network $(G,\bgamma,\vec d)$ has a unique,
up to scaling, faithful realization $(\vec p,\Phi)$, which implies that, for all edges $ij\in E(G)$
\[
\Phi(\gamma_{ij})\cdot \vec p_j - \vec p_i = \alpha_{ij}\vec d_{ij}
\]
for some non-zero scalar $\alpha_{ij}\in \R$.  It follows that, by replacing $\vec d_{ij}$ with
$\Phi(\gamma_{ij})\cdot \vec p_j - \vec p_i$ in the direction realization system
\eqref{colored-crystal-directions} we obtain \eqref{eq:infinitesimal}.  Since $\vec d^\perp$ is
also generic for \propref{dnthmfixedcenter}, we conclude that \eqref{eq:infinitesimal} has full rank
as desired.
\end{proof}

\bibliographystyle{abbrvnat}

\end{document}